\documentclass[english]{article}
\usepackage[T1]{fontenc}
\usepackage[latin9]{inputenc}
\usepackage{xcolor}
\usepackage{pdfcolmk}
\usepackage{babel}

\usepackage{verbatim}
\usepackage{amsthm}
\usepackage{amsmath}
\usepackage{amssymb}
\usepackage{esint}
\PassOptionsToPackage{normalem}{ulem}
\usepackage{ulem}
\usepackage[unicode=true, pdfusetitle,
 bookmarks=true,bookmarksnumbered=true,bookmarksopen=true,bookmarksopenlevel=1,
 breaklinks=false,pdfborder={0 0 1},backref=false,colorlinks=false]
 {hyperref}
\hypersetup{
 pdfstartview=FitH}

\makeatletter

\providecolor{lyxadded}{rgb}{0,0,1}
\providecolor{lyxdeleted}{rgb}{1,0,0}
\theoremstyle{plain}
\newtheorem{thm}{Theorem}[section]
  \theoremstyle{definition}
  \newtheorem{defn}[thm]{Definition}
  \theoremstyle{remark}
  \newtheorem{rem}[thm]{Remark}
  \theoremstyle{plain}
  \newtheorem{lem}[thm]{Lemma}
  \theoremstyle{plain}
  \newtheorem{prop}[thm]{Proposition}
 \theoremstyle{definition}
  \newtheorem{example}[thm]{Example}
  \theoremstyle{plain}
  \newtheorem{cor}[thm]{Corollary}

\def\N{\mathbb{N}}
\def\No{{\mathbb{N}_0}}
\def\R{\mathbb{R}}
\def\Z{\mathbb{Z}}
\def\C{\mathbb{C}}
\def\Nn{{\mathbb{N}^n}}
\def\Non{{\mathbb{N}_0^n}}
\def\Rn{{\mathbb{R}^n}}
\def\Zn{{\mathbb{Z}^n}}
\def\Cn{\mathbb{C}^n}

\makeatother

\begin{document}
\global\long\def\N{\mathbb{N}}
 \global\long\def\No{\mathbb{N}_{0}}
 \global\long\def\R{\mathbb{R}}
 \global\long\def\Z{\mathbb{Z}}
 \global\long\def\C{\mathbb{C}}
 \global\long\def\Nn{\mathbb{N}^{n}}
 \global\long\def\Non{\mathbb{N}_{0}^{n}}
 \global\long\def\Rn{\mathbb{R}^{n}}
 \global\long\def\Zn{\mathbb{Z}^{n}}
 \global\long\def\Cn{\mathbb{C}^{n}}
 \global\long\def\K{\mathbb{K}}

\title{Compactness in quasi-Banach function spaces and applications to compact embeddings of Besov-type spaces\footnote{The research was partially supported by Grant no. P 201 13-14743S of the Grant Agency of the Czech Republic, by RVO: 67985840, by the Center for Research and Development in Mathematics and Applications (CIDMA) and the Portuguese Foundation for Science and Technology (FCT) within Project no. UID/MAT/04106/2013.}}

\author{António Caetano \\ \scriptsize{Center for R\&D in Mathematics and Applications,} \\[-1mm] \scriptsize{Department of Mathematics,
University of Aveiro, 3810-193 Aveiro, Portugal} \\[-1mm] \scriptsize{e-mail: \texttt{acaetano@ua.pt} (corresponding author, tel. +351\ 927994307)}   \and Amiran Gogatishvili \\ \scriptsize{Institute of Mathematics, Academy of Sciences of the Czech Re\-pu\-blic,} \\[-1mm] \scriptsize{\u Zitn\'a 25, 11567 Prague 1, Czech Republic} \\[-1mm] \scriptsize{e-mail: \texttt{gogatish@\allowbreak math.cas.cz}}  \and Bohumír Opic \\ \scriptsize{Department of Mathematical Analysis, Faculty of Mathematics and Physics,  Charles University,} \\[-1mm] \scriptsize{So\-ko\-lovsk\'a 83, 186 75  Prague~8, Czech Republic} \\[-1mm] \scriptsize{e-mail: \texttt{opic@\allowbreak karlin.mff.cuni.cz}} }

\date{}

\maketitle
\begin{abstract}
There are two main aims of the paper. The first one is to extend the criterion for the precompactness of sets in Banach function spaces to the setting of quasi-Banach function spaces. The second one is to extend the criterion for the precompactness of sets in the Lebesgue spaces $L_p(\Rn)$, $1 \leq p < \infty$, to the so-called power quasi-Banach function spaces.
These criteria are applied to establish compact embeddings of abstract Besov spaces into quasi-Banach function spaces. The results are illustrated on embeddings of Besov spaces $B^s_{p,q}(\Rn)$, $0<s<1$, $0<p,q\leq \infty$, into Lorentz-type spaces.

\medskip{}


\medskip{}


\medskip{}

\thanks{\copyright 2016. Licensed under the CC-BY-NC-ND 4.0 license \linebreak http://creativecommons.org/licenses/by-nc-nd/4.0/.}

\thanks{Formal publication: http://dx.doi.org/10.1017/S0308210515000761.}

\end{abstract}

\section{Introduction}

The well-known criterion for the precompactness of sets in a Banach function space states that a subset $K$ of the absolutely continuous part $X_a$ of a Banach function space $X$ is precompact in $X$ if and only if $K$ is locally precompact in measure and $K$ has uniformly absolutely continuous norm (cf. \cite[Chap. 1, Exercise 8]{BS88}).

Such a criterion was, for example, used in \cite{Pus06} to establish compact embeddings $W^kX(\Omega) \hookrightarrow \hookrightarrow Y(\Omega)$. Here $k\in \N$, $\Omega$ is a bounded domain in $\Rn$, $X(\Omega)$ and $Y(\Omega)$ are rearrangement-invariant Banach function spaces and $W^kX(\Omega)$ is the Sobolev space modelled upon the space $X(\Omega)$. Another paper using such a~criterion is, e.g., \cite{RafSam08}, where the authors applied it to get the so-called dominated compactness theorem for regular linear integral operators.

There is a natural question whether the above mentioned criterion characterizing precompact subsets in Banach function spaces can be extended to the setting of quasi-Banach function spaces even when elements of these spaces are not locally integrable (we refer to Section \ref{sec:3} for the definition of quasi-Banach function spaces). The positive answer is given in Theorem \ref{thm:theorem} below. In particular, when the given quasi-Banach function space is the Lebesgue space $L_{p}(\Rn)$ with $0<p<1$, we recover \cite[Lem. 1.1]{KZPS76}.

We also establish an extension of a criterion characterizing precompact sets in the Lebesgue space $L_p(\Rn)$, $1 \leq p < \infty$, (cf., e.g., \cite[Thm. 2.32]{AdaFou05} or \cite[Thm.~IV.8.21]{DS57}) to the case when the space $L_p(\Rn)$ is replaced by a power quasi-Banach function space over $\Rn$ (see Theorem \ref{thm:compactqBFS} and Remark \ref{rem:compactqBFS}; we refer to Definition~\ref{def:powerqBFS} for the notion of a power quasi-Banach function space). 

We apply our criteria to establish compact embeddings of abstract Besov spaces into quasi-Banach function spaces over bounded measurable subsets $\Omega$ of $\Rn$ (see Theorem \ref{thm:compactinL} and Corollary \ref{cor:KsubUAC(Z)}; abstract Besov spaces are introduced in Definition \ref{def:Besov}).

Finally, we illustrate our results on embeddings of Besov spaces $B^s_{p,q}(\Rn)$, $0<s<1$, $0<p,q\leq \infty$, into Lorentz-type spaces over bounded measurable subsets of $\Rn$ (see Theorem \ref{Besovcompact}; we refer to Section \ref{sec:2} for the definition of Lorentz-type spaces).

%The paper is organized as follows. Basic definitions and notations are given in Section \ref{sec:2}. In Section \ref{sec:3} quasi-Banach function spaces and lattices are defined and some of their properties are studied. We conclude this section with a compactness criterion in quasi-Banach function spaces (see Theorem \ref{thm:theorem}). Section \ref{sec:4} is devoted to a compactness criterion in quasi-Banach spaces over $\Rn$. Section \ref{sec:5} 

\section{Notation and Preliminaries}

\label{sec:2}

%
\begin{comment}
Whenever convenient, we use the abbreviation LHS({*}) (RHS({*})) for
the left- (right-) hand side of the relation ({*}).
\end{comment}
{}

For two non-negative expressions $\mathcal{A}$ and $\mathcal{B}$,
the symbol $\mathcal{A}\lesssim\mathcal{B}$ (or $\mathcal{A}\gtrsim\mathcal{B}$)
means that $\mathcal{A}\leq c\mathcal{B}$ (or $c\mathcal{A}\geq\mathcal{B}$
), where $c$ is a positive constant independent of appropriate quantities
involved in $\mathcal{A}$ and $\mathcal{B}$. If $\mathcal{A}\lesssim\mathcal{B}$
and $\mathcal{A}\gtrsim\mathcal{B}$, we write $\mathcal{A}\approx\mathcal{B}$
and say that $\mathcal{A}$ and $\mathcal{B}$ are \emph{equivalent}. 

Given a set $A$, its characteristic function is denoted by $\chi_{A}$.
For $a\in\Rn$ and $r\geq0$, the symbol $B(a,r)$ stands for
the closed ball in $\Rn$ centred at $a$ with the radius $r$. The
notation $|\cdot|_{n}$ is used for Lebesgue measure in $\Rn$.

\newpage
Let $(R,\mu)$ be a measure space (with a non-negative measure $\mu$).\footnote{Here we use notation 
from \cite{BS88}. To be more precise, instead of $(R,\mu)$, one should write $(R,\Sigma,\mu)$, where $\Sigma$ is a $\sigma$-algebra of $\mu$-measurable subsets of $R$.} 
The symbol $\mathcal{M}(R,\mu)$ is used to denote the family of all
complex-valued or extended real-valued $\mu$-measurable functions
defined $\mu$-a.e. on $R$. By $\mathcal{M}^{+}(R,\mu)$ we mean
the subset of $\mathcal{M}(R,\mu)$ consisting of those functions
which are non-negative $\mu$-~a.e. on $R$. If $R$ is a measurable
subset $\Omega$ of $\Rn$ and $\mu$ is the corresponding Lebesgue
measure, we omit the $\mu$ from the notation. Moreover, if $\Omega=(a,b)$,
we write simply $\mathcal{M}(a,b)$ and $\mathcal{M}^{+}(a,b)$ instead
of $\mathcal{M}(\Omega)$ and $\mathcal{M}^{+}(\Omega)$, respectively.
%
\begin{comment}
By $\mathcal{M}^{+}(a,b;\downarrow)$ or $\mathcal{M}^{+}(a,b;\uparrow)$
we mean the collection of all $f\in\mathcal{M}^{+}(a,b)$ which are
non-increasing or non-decreasing on $(a,b)$, respectively
\end{comment}
{} Finally, by $\mathcal{W}(\Omega)$ (or by $\mathcal{W}(a,b)$) we
mean the class of \emph{weight functions} on $\Omega$ (resp. on $(a,b)$),
consisting of all measurable functions which are positive a.e. on
$\Omega$ (resp. on $(a,b)$). A subscript $0$ is added to the previous
notation (as in $\mathcal{M}_{0}(\Omega)$, for example) if in the considered class one restricts to functions which are finite
a.e..

Given two quasi-normed spaces $X$ and $Y$, we write $X=Y$ (and
say that $X$ and $Y$ \emph{coincide}) if $X$ and $Y$ are equal in the
algebraic and the topological sense (their quasi-norms are equivalent).
%
\begin{comment}
The symbol $X\hookrightarrow Y$ or $X\hookrightarrow\hookrightarrow Y$
means that $X\subset Y$ and the natural embedding of $X$ in $Y$
is continuous or compact, respectively. 
\end{comment}
{}

Let $p,q\in(0,\infty]$, let $\Omega$ be a measurable subset of $\Rn$
with $|\Omega|_{n}>0$ and let $w\in\mathcal{W}(0,|\Omega|_{n})$
be such that \begin{eqnarray}
B_{p,q;w}(t): & = & \|\tau^{1/p-1/q}\, w(\tau)\|_{q;(0,t)}\ <\ \infty\quad\text{for all \ensuremath{t\in(0,|\Omega|_{n})}}\label{eq:EGO(2.1)}\\
 &  & \mbox{(and also for \ensuremath{t=|\Omega|_{n}}when \ensuremath{|\Omega|_{n}<\infty})},\nonumber \end{eqnarray}
where $\|\cdot\|_{q;E}$ is the usual $L_{q}$-(quasi-)norm on the
measurable set $E$. The \emph{Lorentz-type space} $L_{p,q;w}(\Omega)$
consists of all (classes of) functions $f\in\mathcal{M}(\Omega)$
for which the quantity \begin{equation}
\|f\|_{p,q;w;\Omega}:=\|t^{1/p-1/q}\, w(t)\, f^{*}(t)\|_{q;(0,|\Omega|_{n})}\label{eq:EGO(2.2)}\end{equation}
is finite; here $f^{*}$ denotes the \emph{non-increasing rearrangement}
of $f$ given by \begin{equation}
f^{*}(t)=\inf\{\lambda>0:|\{x\in\Omega:\ |f(x)|>\lambda\}|_{n}\leq t\},\quad t\geq0.\label{eq:EGO(2.3)}\end{equation}

\noindent We shall also need the inequality (cf. \cite[p. 41]{BS88})
\begin{equation}
(f+g)^*(t) \leq f^*(t/2)+g^*(t/2), \quad t \geq 0,
\label{eq:subadditivity}
\end{equation}
\noindent and the maximal function $f^{**}$ of $f^{*}$ defined
by 
\[
f^{**}(t)=\frac{1}{t}\int_{0}^{t}f^{*}(s)\, ds,\quad t>0.
\]
One can show that the functional \eqref{eq:EGO(2.2)} is a \emph{quasi-norm} on
$L_{p,q;w}(\Omega)$ \emph{if and only if} the function $B_{p,q;w}$
given by \eqref{eq:EGO(2.1)} satisfies 
\begin{equation}
B_{p,q;w}\in\Delta_{2},\label{eq:EGO(2.5)}
\end{equation}
that is, $B_{p,q;w}(2t)\lesssim B_{p,q;w}(t)$ for all $t\in(0,|\Omega|_{n}/2)$. This follows, e.g., from \cite[Cor. 2.2]{CS93}
if $q\in(0,\infty)$. When $q=\infty$, then arguments similar to those used in the proof of \cite[Cor. 2.2]{CS93} together with inequality (\ref{eq:subadditivity}) 
and the fact that
\begin{equation}
\| f \|_{p,\infty;w;\Omega} = \| B_{p,\infty;w}(t) f^*(t) \|_{\infty,(0,|\Omega|_n)} \quad \mbox{for all } f \in L_{p,\infty;w}(\Omega)
\label{eq:extra}
\end{equation}
yield the result. Note also that equality (\ref{eq:extra}) follows on interchanging the essential suprema on its right-hand side.
In particular, one can easily verify that (\ref{eq:EGO(2.5)}) is satisfied provided that
 \[
w(2t)\,\lesssim\, w(t)\quad\text{for a.e. \ensuremath{t\in(0,|\Omega|_{n}/2)}}.\]
 Moreover, since the relation $w\in\mathcal{W}(0,|\Omega|_{n})$ yields
$B_{p,q;w}(t)>0$ for all $t\in(0,|\Omega|_{n})$, one can prove that
the space $L_{p,q;w}(\Omega)$ is \emph{complete} when \eqref{eq:EGO(2.5)} is satisfied
(cf. \cite[Prop.~2.2.9]{CRS07}; if $q=\infty$ one makes use of (\ref{eq:extra}) again). %

If $q\in[1,\infty)$, the spaces $L_{p,q;w}(\Omega)$ are particular
cases of the \emph{classical Lorentz spaces} $\Lambda^{q}(\omega)$.
On the other hand, these Lorentz-type spaces contain as particular
cases a lot of well-known spaces such as Lebesgue spaces $L_{p}(\Omega)$ and
Lorentz spaces $L_{p,q}(\Omega)$, among others.

%
\begin{comment}
and, if $w$ is a \emph{slowly varying function} (see Definition \ref{SV}
below), the so-called \emph{Lorentz-Karamata} spaces $L_{p,q;w}(\Omega)$. 
\end{comment}
{}

If $\Omega=\Rn$, we sometimes omit this symbol in the notation and,
for example, simply write $\|\cdot\|_{p,q;w}$ or $L_{p,q;w}$ instead
of $\|\cdot\|_{p,q;w;\Rn}$ or $L_{p,q;w}(\Rn)$, respectively. 

\inputencoding{latin1}%
\begin{comment}
\selectlanguage{english}%
\inputencoding{latin9}\begin{defn}  
\label{SV} Let $(\alpha,\beta)$ be one of the intervals $(0,\infty)$, $(0,1)$ or $(1,\infty)$. A function $b \in {\cal M}_0^+(\alpha,\beta)$, $0\not\equiv b\not\equiv\infty$, is said to be slowly varying on $(\alpha,\beta)$, notation $b \in SV(\alpha,\beta)$, if, for each $\varepsilon > 0$, there are functions $g_\varepsilon \in {\cal M}_0^+(\alpha,\beta; \uparrow)$ and $g_{-\varepsilon} \in {\cal M}_0^+(\alpha,\beta; \downarrow)$ such that  	\[ 	t^{\varepsilon}b(t) \approx g_\varepsilon(t) \quad \mbox{and} \quad t^{-\varepsilon}b(t) \approx g_{-\varepsilon}(t) \quad \mbox{for all } \ t \in (\alpha,\beta). \] %for all $t \in (\alpha,\beta)$. 
\end{defn} 
\end{comment}
{}\inputencoding{latin9}\foreignlanguage{english}{}%
\begin{comment}
\selectlanguage{english}%
 If $b \in SV(0,1)$, then we assume that $b$ is extended by 1 in the interval $[1,\infty)$.
\end{comment}
{}

\selectlanguage{english}%
%
\begin{comment}
\medskip
Given $f \in L_p$, $1\leq p<\infty$, the {\it first difference operator} $\Delta_h$ of step $h \in \Rn$ transforms $f$ in $\Delta_h f$ defined by  	\[ 	(\Delta_h f)(x) := f(x+h) - f(x), \quad x \in \Rn, \] whereas the {\it modulus of continuity} of $f$ is given by     \[     \omega_1(f,t)_p := \sup_{{h \in \Rn} \atop {| h | \leq t}} \| \Delta_h f \|_p, \quad t>0. \]
\medskip
\begin{defn} \label{besov} Let $1 \leq p < \infty$, $1 \leq r \leq \infty$  and let $b \in SV(0,1)$ be such that  \begin{equation} 	\label{(1)} 	\| t^{-1/r} b(t) \|_{r,(0,1)} = \infty. \end{equation} The Besov space $B^{0,b}_{p,r} = B^{0,b}_{p,r}(\R^n)$ consists of those functions $f \in L_p$ for which the norm \begin{equation}  \label{BN}      \| f \|_{B^{0,b}_{p,r}}:= \| f \|_p + \| t^{-1/r} b(t)\, \omega_1(f,t)_p \|_{r,(0,1)} \end{equation} is finite. \end{defn} 

Finally, for $b\in SV(0,1)$ (with $b(t)=1$ if $t\in[1,2)$) and
$1\leq r\leq\infty,$ define

\begin{equation}
b_{r}(t):=\|s^{-1/r}b(s^{1/n})\|_{r,(t,2)}.\label{eq:b_r}\end{equation}
 and, given also $1\leq p<\infty$, $0<q\leq\infty$,

\begin{equation}
\tilde{b}(t):=\left\{ \begin{array}{ll}
b_{r}(t)^{1-r/q+r/\max\{p,q\}}b(t^{1/n})^{r/q-r/\max\{p,q\}} & \mbox{ if }\, r\not=\infty\\
b_{\infty}(t) & \mbox{ if }\, r=\infty\end{array}\right..\label{eq:b tilde}\end{equation}

\end{comment}
{}

\section{A compactness criterion in quasi-Banach function spaces}

\label{sec:3}

In what follows, the symbol $(R,\mu)$ stands for a totally $\sigma$-finite
measure space and in $\mathcal{M}_{0}=\mathcal{M}_{0}(R,\mu)$ we consider the topology of convergence in measure on sets of finite measure
(see \cite[p.~3]{BS88}), which we briefly refer to as the topology
of local convergence in measure.
\begin{defn}
\label{def:wfq-n}A mapping $\rho:\mathcal{M}^{+}(R,\mu)\rightarrow[0,\infty]$
is called a \emph{function quasi-norm} if there exists a constant
$C\in[1,\infty)$ such that, for all $f$, $g$, $f_{k}$ ($k\in\N)$
in $\mathcal{M}^{+}(R,\mu)$, for all constants $a\geq0$ and for
all $\mu$-measurable subsets $E$ of $R$, the following properties
hold:

\textbf{(P1)} $\rho(f)=0\Leftrightarrow f=0\ \mu$-a.e.; \quad{}$\rho(af)=a\rho(f)$;\quad{}$\rho(f+g)\leq C\,(\rho(f)+\rho(g))$; 

\textbf{(P2)} $g\leq f\;\mu$-a.e. $\Rightarrow\;\rho(g)\leq\rho(f)$;

\textbf{(P3)} $f_{k}\uparrow f\;\mu$-a.e. $\Rightarrow\;\rho(f_{k})\uparrow\rho(f)$;

\textbf{(P4)} $\mu(E)<\infty\Rightarrow\rho(\chi_{E})<\infty$.

\end{defn}
\smallskip{}

\begin{defn}
\label{def:wq-BFS}Let $\rho:\mathcal{M}^{+}(R,\mu)\rightarrow[0,\infty]$
satisfy properties \textbf{(P1)-(P3)} of Definition
\ref{def:wfq-n}. The collection $X=X(\rho)$ of all functions $f\in\mathcal{M}(R,\mu)$
for which $\|f\|_{X}:=\rho(|f|)<\infty$ is called a \emph{quasi-Banach
function lattice} (\emph{q-BFL}, for short) over $(R,\mu)$ (or, simply,
over $R$ if $\mu$ is clearly meant). If, in addition, $\rho$ is a function
quasi-norm, then we also call $X(\rho)$ a \emph{quasi-Banach
function space }(\emph{q-BFS}, for short) over $(R,\mu)$ (or, simply,
over $R$).\end{defn}

In what follows we shall use the fact that in any quasi-normed linear space $(X,\|\cdot\|_X)$ there is a $\lambda$-norm $|||\cdot|||$ (with $\lambda \in (0,1]$ satisfying $(2C)^\lambda=2$, where $C$ is from Definition \ref{def:wfq-n}) such that 
\begin{equation}
|||f|||\approx\|f\|_{X}\quad \text{for all}\ \ f \in X %\forall f\in X
\label{eq:equiv}
\end{equation}
-- cf., e.g., \cite[p. 2]{ET96} and \cite[p. 20]{DeVoLo93}. This result goes back to \cite{Aoki42} and \cite{Ro57} -- see also \cite[pp. 66-67]{Pie07}. 
Recall that the $\lambda$-norm $|||\cdot|||$ , $\lambda \in (0,1]$, satisfies, for all $f, g \in X$ and all scalars $\alpha$,
\[ |||f||| = 0 \quad\text {if and only if} \quad f=0,\]
\[ |||\alpha f||| = |\alpha|\ |||f|||,\] 
\begin{equation}\label{triangleineq}
|||f + g|||^{\lambda} \le |||f|||^{\lambda} + |||g|||^{\lambda} . 
\end{equation}
\begin{lem}
\label{star}
Let $X=X(\rho)$ be a q-BFL. Then $X \subset \mathcal{M}_0$ and under the natural vector space operations $(X,\|\cdot\|_{X})$ is indeed a quasi-normed linear space. Moreover, $X \hookrightarrow \mathcal{M}_0$. In particular, if $f_k\underset{k}{\longrightarrow} f$ in $X$, then $f_k\underset{k}{\longrightarrow} f$ locally in measure, and hence some subsequence converges pointwise $\mu$-a.e. to $f$.
\end{lem}

\begin{proof}
Given any $f \in X$, the set $A$ in which $f$ is infinite has a measure 0. Indeed, as $N \chi_A \leq |f|$ for any $N \in \N$, properties \textbf{(P1)} and \textbf{(P2)} imply that
$$N \rho(\chi_A) = \rho(N \chi_A) \leq \rho(|f|) < \infty,$$
and thus $\rho(\chi_A)=0$. Together with \textbf{(P1)}, this shows that $\chi_A=0$ $\mu$-a.e.. 

%Therefore, 
Since $\mu(A)=0$, $X$ is a subspace of $\mathcal{M}_0$ and so $X$ inherits the vector space operations from $\mathcal{M}_0$ (where, as usual, any two functions coinciding $\mu$-a.e.
are identified). Moreover, by \textbf{(P1)}, the space $(X,\|\cdot\|_{X})$ is a quasi-normed linear space.

It remains to prove the continuous embedding $X \hookrightarrow \mathcal{M}_0$. This can be done using some ideas of \cite[Chap. II, Thm. 1, pp. 41-42]{KPS82}. However, since our setting and that of \cite{KPS82} are different, we prove it here for the convenience of the reader.

It is sufficient to show that the condition $\| f_k \|_X \underset{k}{\longrightarrow} 0$ implies the convergence of $\{f_k\}_k$ to zero in $\mathcal{M}_0$. 

On the contrary, assume that there are a set $E$, with $0<\mu(E)<\infty$, and \linebreak $\varepsilon >0$ such that $\mu\{x \in E : \; |f_k(x)|>\varepsilon\}$ fails to converge to 0 as $k \to \infty$. Then there exists $\delta>0$ and a subsequence $\{f_{\sigma(k)}\}_k$ such that the inequalities $|f_{\sigma(k)}(x)|>\varepsilon$ hold on sets $E_k \subset E$ satisfying $\mu(E_k)>\delta$ for all $k \in \N$. Hence, $\varepsilon \chi_{E_k}(x) \leq |f_{\sigma(k)}(x)|$, and, on using \textbf{(P2)} and \textbf{(P1)}, we obtain that $\varepsilon\| \chi_{E_k} \|_X \leq \| f_{\sigma(k)} \|_X$ for all $k \in \N$. Without loss of generality, we may assume that the sequence $\{\sigma(k)\}_k$ is chosen so that $\sum_{k=1}^\infty \| f_{\sigma(k)} \|_X^\lambda < \infty$, where $\lambda \in (0,1]$ corresponds to the $\lambda$-norm $|||\cdot|||$ considered in \eqref{eq:equiv}. Consequently, 
\begin{eqnarray}
\lefteqn{\varepsilon^\lambda \Big\| \sum_{k=1}^N \chi_{E_k} \Big\|_X^\lambda \approx \varepsilon^\lambda \Big|\Big|\Big| \sum_{k=1}^N \chi_{E_k} \Big|\Big|\Big|^\lambda \; \lesssim \; \varepsilon^\lambda \sum_{k=1}^N \| \chi_{E_k} \|_X^\lambda \nonumber} \nonumber \\
& \leq &  \sum_{k=1}^N \| f_{\sigma(k)} \|_X^\lambda \; \leq \; \sum_{k=1}^\infty \| f_{\sigma(k)} \|_X^\lambda \; < \; \infty \ \  \text{for all}\ N \in \N
\label{eq:star}
\end{eqnarray}
(we emphasize that constants hidden in symbols $\,\approx\,$ and $\,\lesssim\,$ are independent of $N$).
Thus, if we show that 
\begin{equation}
\Big\| \sum_{k=1}^N \chi_{E_k} \Big\|_X \underset{N}{\longrightarrow} \infty,
\label{eq:2star}
\end{equation}
we arrive at a contradiction and the proof will be complete.

In order to prove \eqref{eq:2star}, it is enough to verify that
\begin{equation}
\sum_{k=1}^N \chi_{E_k} \uparrow f \quad \mbox{with} \quad f \notin \mathcal{M}_0.
\label{eq:3star}
\end{equation}
Indeed, \eqref{eq:3star} and \textbf{(P3)} imply that $\| \sum_{k=1}^N \chi_{E_k} \|_X \uparrow \rho(f)$ and, since $f \notin \mathcal{M}_0 \supset X$, $\rho(f)=\infty$.
As it is obvious that the sequence $\{ \sum_{k=1}^N \chi_{E_k} \}_N$ is non-decreasing, all that remains to prove in order to establish \eqref{eq:3star} is that 
$f \notin \mathcal{M}_0$.

Again we proceed by contradiction and assume that $f \in \mathcal{M}_0$. Recalling that $E_k \subset E$, $k \in \N$, and $\mu(E)<\infty$, we see that we can use Egorov's Theorem to state that there exists a set $E' \subset E$, with $\mu(E \setminus E')<\frac{\delta}{2}$, on which the convergence of $\sum_{k=1}^N \chi_{E_k}$ to $f$ is uniform. As a consequence, the boundedness of each $\sum_{k=1}^N \chi_{E_k}$ on $E'$ implies that $f$ is bounded on $E'$, too. Thus,
\begin{equation}
\int_{E'} f\, d\mu < \infty.
\label{eq:mirek*}
\end{equation}
However, since the inequalities $\mu(E_k \setminus E')<\frac{\delta}{2}$ and $\mu(E_k)>\delta$ imply that $\mu(E'\cap E_k)>\frac{\delta}{2}$, $k \in \N$, we arrive at
\begin{equation*}
\int_{E'} f\, d\mu \geq \int_{E'} \sum_{k=1}^N \chi_{E_k}\, d\mu = \sum_{k=1}^N \mu(E'\cap E_k) > N \frac{\delta}{2}
\end{equation*}
for all $N \in \N$, which contradicts \eqref{eq:mirek*}.
\end{proof}

\begin{rem}
\label{rem:axioms}(i) In contrast to \cite[p.~9]{Sal08}, in our definition of a q-BFS we do not require {\em a priori} that $X \hookrightarrow \mathcal{M}_0$ and that $X$ is complete since these two properties are consequences of axioms \textbf{(P1)-(P4)}(cf. Lemma \ref{star} and Lemma~\ref{lem:completude} below).

(ii) Recall that (cf. \cite{BS88}) a \emph{function norm} is a mapping
$\rho:\mathcal{M}^{+}(R,\mu)\to[0,\infty]$ satisfying \textbf{(P1)}
with $C=1$, \textbf{(P2)-(P4)} and 

\textbf{(P5)} $\mu(E)<\infty\Rightarrow\int_{E}f\, d\mu\leq c_{E}\,\rho(f)$
for some constant $c_{E}$, $0<c_{E}<\infty$, depending on $E$
and $\rho$ but independent of $f\in\mathcal{M}^{+}(R,\mu)$. 

\noindent Thus, any function norm is
a function quasi-norm. Hence, taking a function norm $\rho$ and defining
the Banach function space (BFS) as the family of those $f\in\mathcal{M}(R,\mu)$
for which $\rho(|f|)<\infty$, we see that any BFS is a q-BFS. 
\end{rem}

The first example of q-BFS (and, \emph{a fortiori}, of q-BFL) is the
Lebesgue space $L_{p}(\Omega)$ with $0<p\leq\infty$, which is
also a BFS when $p\geq1$. The next example of q-BFS is the Lorentz-type
space $L_{p,q;w}(\Omega)$ %
 introduced in Section \ref{sec:2} provided that \eqref{eq:EGO(2.1)}
and \eqref{eq:EGO(2.5)} hold. (In this case $R=\Omega$, $\mu$ is
the Lebesgue measure on $\Omega$ and $\rho=\|\cdot\|_{p,q;w;\Omega}$.)
Indeed, property \textbf{(P1)} follows from what has been said in
Section \ref{sec:2} and properties \textbf{(P2)-(P4)}
are easy consequences of the definition of such
spaces. 

\textbf{\smallskip{}
}

The next result follows by the same arguments as \cite[Chap.~1, Lem.~1.5]{BS88}.
\begin{lem}
\label{lem:fatou}Let $X=X(\rho)$ be a q-BFL and assume that $f_{k}\in X$,
$k\in\N$. 

\textbf{\emph{(i)}} If $0\leq f_{k}\uparrow f\;\mu$-a.e., then either
$f\notin X$ and $\|f_{k}\|_{X}\uparrow\infty$, or $f\in X$ and
$\|f_{k}\|_{X}\uparrow\|f\|_{X}$. 

\textbf{\emph{(ii)}} \emph{(Fatou's lemma)} If $f_{k}\rightarrow f\;\mu$-a.e.,
and if $\liminf_{k\to\infty}\|f_{k}\|<\infty$, then $f\in X$ and
$\|f\|_{X}\leq\liminf_{k\to\infty}\|f_{k}\|_{X}$.
\end{lem}
We shall also need the following counterpart of \cite[Chap. 1, Thm. 1.6]{BS88},
where the number $\lambda\in(0,1]$ corresponds to the $\lambda$-norm $|||\cdot|||$ appearing in \eqref{eq:equiv}.

\begin{lem}
\label{lem:completude} Let $X=X(\rho)$ be a q-BFL. Assume that $f_{k}\in X$,
$k\in\N$, and that\begin{equation}
\sum_{k=1}^{\infty}\|f_{k}\|_{X}^{\lambda}<\infty.\label{eq:hypoth}\end{equation}
Then $\sum_{k=1}^{\infty}f_{k}$ converges in $X$ to a function $f\in X$
and\begin{equation}
\|f\|_{X}\lesssim\left(\sum_{k=1}^{\infty}\|f_{k}\|_{X}^{\lambda}\right)^{1/\lambda}.\label{eq:domination}\end{equation}
In particular, $X$ is complete.\end{lem}
\begin{proof}
If $t=\sum_{k=1}^{\infty}|f_{k}|$, $t_{K}=\sum_{k=1}^{K}|f_{k}|$,
$K\in\N$, then $0\leq t_{K}\uparrow t$. Since, by \eqref{eq:equiv} and \eqref{triangleineq},
\[
\|t_{K}\|_{X}\lesssim\left(\sum_{k=1}^{K}\|f_{k}\|_{X}^{\lambda}\right)^{1/\lambda}\quad \text{for all}\ \ K\in\N,\]
it follows from \eqref{eq:hypoth} and Lemma \ref{lem:fatou}(i) that
$t\in X$. In particular, since $X \subset \mathcal{M}_0$, $\sum_{k=1}^{\infty}|f_{k}(x)|$
converges pointwise $\mu$-a.e., and hence so does $\sum_{k=1}^{\infty}f_{k}(x)$.
Thus, if\[
f:=\sum_{k=1}^{\infty}f_{k},\quad s_{K}:=\sum_{k=1}^{K}f_{k},\quad K\in\N,\]
then $s_{K}\to f$ $\,\mu$-a.e.. Hence, given any $M\in\N$, $s_{K}-s_{M}\to f-s_{M}$
$\,\mu$-a.e. as $K\to\infty$. Moreover, \[
\liminf_{K\to\infty}\|s_{K}-s_{M}\|_{X}\lesssim\liminf_{K\to\infty}\left(\sum_{k=M+1}^{K}\|f_{k}\|_{X}^{\lambda}\right)^{1/\lambda}=\left(\sum_{k=M+1}^{\infty}\|f_{k}\|_{X}^{\lambda}\right)^{1/\lambda},\]
which tends to 0 as $M\to\infty$, because of hypothesis \eqref{eq:hypoth}.
Then, by Fatou's lemma (Lemma \ref{lem:fatou}(ii)), $f-s_{M}\in X$
(therefore also $f\in X$) and $\|f-s_{M}\|_{X}\to0$ as $M\to\infty$.
Finally,\[
\|f\|_{X}=\|f-s_{M}+\sum_{k=1}^{M}f_{k}\|_{X}\lesssim\left(\|f-s_{M}\|_{X}^{\lambda}+\sum_{k=1}^{M}\|f_{k}\|_{X}^{\lambda}\right){}^{1/\lambda}\]
and \eqref{eq:domination} follows by letting $M$ tend to infinity.

Completeness follows by standard arguments as in the normed case.
%However, for the convenience of the reader, we provide a proof here. 
%
%Consider a Cauchy sequence $\{f_k\}_{k}$ in $X$. Then there exists
%a subsequence $\{g_{j}\}_{j}$, $g_{j}:=f_{k_{j}}$, $j\in\N$, such
%that\[
%\|g_{j+1}-g_{j}\|_{X}^\lambda\leq2^{-j}\quad\mbox{for all }\, j\in\N.\]
%Now\[
%g_{j}=g_{1}+(g_{2}-g_{1})+\ldots+(g_{j}-g_{j-1}),\quad2\leq j,\]
%and\[
%\sum_{j=1}^{\infty}\|g_{j+1}-g_{j}\|_{X}^{\lambda}\leq\sum_{j=1}^{\infty}2^{-j}=1.\]
%Thus, by the already proved part of Lemma \ref{lem:completude}, the
%partial sums $\sum_{j=1}^{M}(g_{j+1}-g_{j})=g_{M+1}-g_{1}$ converge
%in $X$ to some $g\in X$. Then also $g_{M+1}\underset{M}{\longrightarrow}g+g_{1}=:f\in X$.
%Since $\{g_{j}\}_{j}$ is a subsequence of the Cauchy sequence $\{f_{k}\}_{k}$,
%it is easy to see that also $f_{k}\underset{k}{\longrightarrow}f$
%in $X$. 
%
\end{proof}
As a consequence of this lemma, a q-BFL is a quasi-Banach space
(i.e., a~complete quasi-normed space).
\begin{defn}
A sequence $\{E_{n}\}_{n}$ of $\mu$-measurable subsets of a measure
space $(R,\mu)$ is said to converge $\mu$-a.e. to the empty set
(notation $E_{n}\underset{n}{\longrightarrow}\emptyset$ $\,\mu$-a.e.)
if $\cap_{n=1}^{\infty}\cup_{m=n}^{\infty}E_{m}$ is a set of $\mu$-measure
zero or, equivalently, if $\chi_{E_{n}}\underset{n}{\longrightarrow}0$
$\mu$-a.e.
\end{defn}
\smallskip{}

\begin{defn}
\label{def:abscont}A function $f$ in a q-BFL $X$ is said to have
\emph{absolutely continuous quasi-norm} if $\|f\chi_{E_{n}}\|_{X}\underset{n}{\longrightarrow}0$
for every sequence $\{E_{n}\}_{n=1}^{\infty}$ such that $E_{n}\underset{n}{\longrightarrow}\emptyset$
$\mu$-a.e.. The set of all functions in $X$ which have absolutely
continuous quasi-norm is denoted by $X_{a}$. If $X_{a}=X$, then
the space $X$ is said to have \emph{absolutely continuous
quasi-norm.}
\end{defn}
Analogously to the BFS case, one can prove the next assertion (cf. \cite[Chap. 1, Prop. 3.2]{BS88}.
\begin{prop}
\label{pro:decreasing} A function $f$ in a q-BFL $X$ has absolutely
continuous quasi-norm if and only if $\|f\chi_{E_{n}}\|_{X}\downarrow0$
for every sequence $\{E_{n}\}_{n=1}^{\infty}$ such that $E_{n}\downarrow\emptyset$
$\mu$-a.e. $($which means, besides $E_{n}\underset{n}{\longrightarrow}\emptyset$
$\mu$-a.e., that the sequence $\{E_{n}\}_{n=1}^{\infty}$ is non-increasing$)$.\end{prop}
\begin{example}
If $0<p\leq \infty$, $0<q<\infty$ and the weight $w$ satisfies \eqref{eq:EGO(2.1)}
and \eqref{eq:EGO(2.5)}, then the space $L_{p,q;w}(\Omega)$ has absolutely continuous quasi-norm. This
follows immediately from Proposition~\ref{pro:decreasing} and the
Lebesgue dominated convergence theorem.
\end{example}
%
\begin{comment}
Warning: What follows is a proof for the definition we have used formerly.
If a proof is to be really inserted in the file, adapt the proof below. 

Assume not. Then, for some $\varepsilon>0$, there are measurable
sets $E_{k}$, $k\in\N$, satisfying $|E_{k}|<1/k$, $\|f\chi_{E_{k}}\|_{X}\geq\varepsilon$.
On the other hand, \[
\|f\chi_{E_{k}}\|_{X}\leq\left(\int_{0}^{|\Omega|}t^{q/p-1}w(t)^{q}f^{*}(t)^{q}\chi_{[0,|E_{k}|)}(t)\, dt\right)^{1/q},\]
 so that, by the Lebesgue dominated convergence theorem, $\|f\chi_{E_{k}}\|_{X}\underset{k}{\to}0$
and we get a contradiction.
\end{comment}
{}
\begin{prop}
\label{pro:closed}The set $X_{a}$ of functions from the q-BFL $X$
with absolutely continuous quasi-norm is a closed linear subspace
of $X$.\end{prop}
%\begin{proof}
%That $X_{a}$ is a linear subspace of $X$ is evident. The proof that
%$X_{a}$ is closed is similar to the corresponding one in the setting
%of BFS. We prove it here for the convenience of the reader. The number
%$C$ below is the constant from Definition \ref{def:wfq-n}.
%
%If $f_{n}\in X_{a}$ ($n\in\mathbb{N}$) with $f_{n}\underset{n}{\longrightarrow}f$
%in $X$, then, given any $\varepsilon>0$, there exists $N\in\mathbb{N}$
%such that $\|f-f_{N}\|_{X}<\varepsilon/(2C)$. Let $\{E_{m}\}_{m=1}^{\infty}$
%be a sequence satisfying $E_{m}\downarrow\emptyset$ $\mu$-a.e..
%Since $f_{N}\in X_{a}$, there exists $M\in\mathbb{N}$ such that
%$\|f_{N}\chi_{E_{m}}\|_{X}<\varepsilon/(2C)$ for any $m\geq M$.
%Hence, for such $m$, \begin{eqnarray*}
%\|f\chi_{E_{m}}\|_{X} & \leq & C\,(\|(f-f_{N})\chi_{E_{m}}\|_{X}+\|f_{N}\chi_{E_{m}}\|_{X})\\
 %& \leq & C\,\|f-f_{N}\|_{X}+C\,\|f_{N}\chi_{E_{m}}\|_{X}\\
 %& < & \frac{\varepsilon}{2}+\frac{\varepsilon}{2}\ =\ \varepsilon.\end{eqnarray*}
%Therefore, $\|f\chi_{E_{m}}\|_{X}\downarrow0$ and hence $f\in X_{a}$.\end{proof}
The proof follows by the same arguments as those of \cite[Chap. 1, Thm. 3.8]{BS88}.
\begin{defn}
\label{def:unifabscont}Let $X$ be a q-BFL and let $K\subset X_{a}$.
Then $K$ is said to have \emph{uniformly absolutely continuous quasi-norm
}in $X$ (notation $K\subset\mbox{UAC}(X)$) if, for every sequence
$\{E_{n}\}_{n=1}^{\infty}$ with $E_{n}\underset{n}{\longrightarrow}\emptyset$
$\mu$-a.e., \[
\forall\varepsilon>0,\,\exists N\in\mathbb{N}:\; f\in K,\, n\geq N\,\Rightarrow\|f\chi_{E_{n}}\|_{X}<\varepsilon.\]
\end{defn}

As in Definition \ref{def:abscont}, it is irrelevant in %indifferent in
Definition \ref{def:unifabscont} whether we consider all sequences
of sets satisfying $E_{n}\underset{n}{\longrightarrow}\emptyset$ or $E_{n}\downarrow\emptyset$ $\mu$-a.e..

\bigskip

The following assertion is a consequence of Definition \ref{def:unifabscont}.
\begin{prop}
\label{pro:UACimplies}If $K\subset\mbox{UAC}\,(X)$, then 
\begin{equation}
\forall\varepsilon>0,\exists\delta>0:\, f\in K,\,\mu(E)<\delta\ \Rightarrow\,\|f\chi_{E}\|_{X}<\varepsilon.
\label{eq:simplerUAC}
\end{equation}
\end{prop}
\begin{proof}
Assume that $K\subset\mbox{UAC}(X)$ and that (\ref{eq:simplerUAC})
is false. Then there exists $\varepsilon>0$ such that for any $n\in\mathbb{N}$ there exists
$f_{n}\in K$ and $E_{n}\subset R$ satisfying \begin{equation}
\mu(E_{n})<2^{-n}\quad\mbox{and}\quad\|f_{n}\chi_{E_{n}}\|_{X}\geq\varepsilon.\label{eq:contradiction}\end{equation}
Hence, $\mu(\cup_{n=m}^{\infty}E_{n})\leq\sum_{n=m}^{\infty}\mu(E_{n})<2^{-m+1}$,
and thus $E_{n}\to\emptyset$ $\mu$-a.e.. Together with
the fact that $K\subset\mbox{UAC}(X)$, this shows that there is $N\in\mathbb{N}$
such that $\|f_{n}\chi_{E_{n}}\|_{X}<\varepsilon$ for all $n\geq N$.
However, this contradicts \eqref{eq:contradiction} and the result
follows.\end{proof}
\begin{rem}
\label{rem:converse}(i) When $R$ has finite measure, the converse
of the preceding result is also true.

(ii) For the sake of completeness, let us mention that, similarly
to Proposition~\ref{pro:UACimplies}, any function $f$ with \emph{absolutely
continuous quasi-norm} in a q-BFL $X$ satisfies
\[
\forall\varepsilon>0,\exists\delta>0:\,\mu(E)<\delta\ \Rightarrow\,\|f\chi_{E}\|_{X}<\varepsilon.
\]
 The converse is also true when $R$ has finite measure.\end{rem}
\begin{lem}
\label{lem:lemma}Let $X$ be a q-BFS. The sequence $\{f_{k}\}_{k}\subset X_{a}$
is convergent in $X$ if and only if $\{f_{k}\}_{k}$ converges locally
in measure and $\{f_{k}\}_{k}\subset\mbox{UAC}(X)$. Moreover, in
the case of convergence the two limits coincide.\end{lem}
\begin{proof}
Assume that $\{f_{k}\}_{k}\subset X_{a}$ and that it converges in quasi-norm to $f\in X$.
Then it follows from Lemma \ref{star}
that $\{f_{k}\}_{k}$ converges locally in measure to $f$.

We now show that $\{f_{k}\}_{k}\subset\mbox{UAC}(X)$. Let $\varepsilon>0$
and consider any $\{E_{m}\}_{n=1}^{\infty}$ with $E_{m}\underset{m}{\longrightarrow}\emptyset$
$\mu$-a.e.. As $f_{k}\underset{k}{\longrightarrow}f$ in $X$, there
exists $N\in\N$ such that \begin{equation}
\|f_{k}-f\|_{X}<\frac{\varepsilon}{3C^{2}}\quad\mbox{ for all }\, k\geq N,\label{eq:conv}\end{equation}
 where $C$ is the constant from Definition \ref{def:wfq-n}. Since
$\{f_{k}\}_{k}\subset X_{a}$, there exists $M\in\mathbb{N}$ such
that \[
\|f_{k}\chi_{E_{m}}\|_{X}<\frac{\varepsilon}{3C^{2}}<\varepsilon\quad\mbox{for all }\, m\geq M\,\mbox{ and }\, k=1,\ldots,N.\]
 Together with \eqref{eq:conv} and properties \textbf{(P1)} and \textbf{(P2)} this implies
that \begin{eqnarray*}
\|f_{k}\chi_{E_{m}}\|_{X} & \leq & C\|(f_{k}-f)\chi_{E_{m}}\|_{X}+C\|f\chi_{E_{m}}\|_{X}\\
 & \leq & C\|f_{k}-f\|_{X}+C^{2}\|(f-f_{N})\chi_{E_{m}}\|_{X}+C^{2}\|f_{N}\chi_{E_{m}}\|_{X}\\
 & < & \frac{\varepsilon}{3C}+\frac{\varepsilon}{3}+\frac{\varepsilon}{3}\:\leq\:\varepsilon\quad\mbox{for all }m\geq M\,\mbox{ and }\, k>N.\end{eqnarray*}
Thus, $\{f_{k}\}_{k}\subset\mbox{UAC}(X)$.

Conversely, assume now that $\{f_{k}\}_{k}$ converges locally in
measure to some function $f\in\mathcal{M}_{0}(R,\mu)$ and that $\{f_{k}\}_{k}\subset\mbox{UAC}(X)$.
We start by observing that the first of these two hypotheses guarantees
that $f_{\sigma(k)}\underset{k}{\longrightarrow}f\:\mu$-a.e. for some
subsequence $\{f_{\sigma(k)}\}_{k}$ of $\{f_{k}\}_{k}$. Let $\varepsilon>0$. Since our measure space $(R,\mu)$ is totally $\sigma$-finite,
there is a non-decreasing sequence of sets $F_{n}$ such that $\cup_{n=1}^{\infty}F_{n}=R$
and $\mu(F_{n})<\infty$ for all $n \in N$. Without loss of generality, we can assume
that $0<\mu(F_{n})$, $n \in N$. Then $(R\setminus F_{n})\underset{n}{\longrightarrow}\emptyset$
$\mu$-a.e. and the
hypothesis $\{f_{k}\}_{k}\subset\mbox{UAC}(X)$ implies that there
exists $N\in\mathbb{N}$ such that
$\|f_{k}\chi_{R\setminus F_{n}}\|_{X}<\frac{\varepsilon}{6C^{2}}$ for all $k$ and all $n\geq N$.
Hence, on putting $P:=F_{N}$, we have $\mu(P)<\infty$ and 
\begin{equation}
\|f_{k}\chi_{R\setminus P}\|_{X}<\frac{\varepsilon}{6C^{2}}\quad\mbox{for all \,}k\in\mathbb{N}.\label{eq:RminusPk}
\end{equation}
Then Fatou's lemma applied to $f_{\sigma(k)}\chi_{R\setminus P}$ implies that also \begin{equation}
f\chi_{R\setminus P}\in X\quad\mbox{and}\quad\|f\chi_{R\setminus P}\|_{X}\leq\frac{\varepsilon}{6C^{2}}.\label{eq:RminusP}\end{equation}

On the other hand, together with property \textbf{(P4)} applied to $P$, the convergence of $\{f_{k}\}_{k}$ locally in
measure to $f$ guarantees that \begin{equation}
\mu(E_{j})\underset{j}{\longrightarrow}0,\quad\mbox{where}\; E_{j}:=\{x\in P:|f_{j}(x)-f(x)|>\frac{\varepsilon}{3C^{2}\rho(\chi_{P})}\}.\label{eq:defEj}\end{equation}
Now the hypothesis that $\{f_{k}\}_{k}\subset\mbox{UAC}(X)$, Proposition
\ref{pro:UACimplies} and \eqref{eq:defEj} imply that there exists
$J\in\mathbb{N}$ such that \begin{equation}
\|f_{k}\chi_{E_{j}}\|_{X}<\frac{\varepsilon}{6C^{3}}\quad\mbox{for all }\, k\in\mathbb{N}\,\mbox{ and all }\, j\geq J.\label{eq:Ejk}\end{equation}
Thus, another application of Fatou's lemma, to $f_{\sigma(k)}\chi_{E_{j}}$,
allows us to conclude that also\begin{equation}
f\chi_{E_{j}}\in X\quad\mbox{and}\quad\|f\chi_{E_{j}}\|_{X}\leq\frac{\varepsilon}{6C^{3}}\quad\mbox{for all }\, j\geq J.\label{eq:Ej}\end{equation}
 Estimates \eqref{eq:RminusPk}-\eqref{eq:Ej} and properties \textbf{(P1)}
and \textbf{(P2) }imply that\begin{eqnarray*}
\rho(|f_{k}-f|) & \leq & C\,\|(f_{k}-f)\chi_{R\setminus P}\|_{X}+C^{2}\|(f_{k}-f)\chi_{E_{k}}\|_{X}+C^{2}\rho(|f_{k}-f|\chi_{P\setminus E_{k}})\\
 & < & C^{2}\,(\frac{\varepsilon}{6C^{2}}+\frac{\varepsilon}{6C^{2}})+C^{3}\,(\frac{\varepsilon}{6C^{3}}+\frac{\varepsilon}{6C^{3}})+C^{2}\frac{\varepsilon}{3C^{2}\rho(\chi_{P})}\rho(\chi_{P})\\
 & = & \frac{\varepsilon}{3}+\frac{\varepsilon}{3}+\frac{\varepsilon}{3}\:=\:\varepsilon\quad\mbox{for all }\, k\geq J.\end{eqnarray*}
 Therefore, $f\in X$ and $\{f_{k}\}_{k}$ converges to $f$ in $X$.\end{proof}
\begin{rem}
\label{rem:not nec}(i) The hypothesis $\{f_{k}\}_{k}\subset X_{a}$
was not necessary in the proof of the assertion that convergence in quasi-norm implies
convergence in measure.

(ii) This lemma with $\mu(R)<\infty$ for the Orlicz norm was proved in \cite[Lem.~11.2]{KR61}.

\medskip{}

As already mentioned in the Introduction, in the context of Banach function spaces the following theorem is
contained in \cite[Chap. 1, Exercise 8]{BS88}.\end{rem}
\begin{thm}
\label{thm:theorem}Let $X$ be a q-BFS and $K\subset X_{a}$. Then
$K$ is precompact in $X$ if and only if it is locally precompact
in measure and $K\subset\mbox{UAC}\,(X)$.\end{thm}
\begin{proof}
Since the topologies involved are metrizable, we can prove precompactness
using the notion of sequential precompactness.

The sufficiency follows immediately from Lemma~\ref{lem:lemma}.

As for the necessity, observe that if $K$ is precompact in the space $X$, then, by Lemma~\ref{lem:lemma},
it is locally precompact in measure. Thus, it only remains to show that $K\subset\mbox{UAC}(X)$.
Assume that it is not the case. Then there exists a sequence of sets $E_{n}\downarrow\emptyset$
$\mu$-a.e. and $\varepsilon>0$ such that for each $k\in\N$ there
exists a function $f_{k}\in K$ satisfying \begin{equation}
\|f_{k}\chi_{E_{k}}\|_{X}\geq\varepsilon.\label{eq:cont}\end{equation}
On the other hand, by the precompactness of $K$ in $X$, there is
a subsequence $\{f_{\sigma(k)}\}_{k}$ which converges in $X$, say,
to $f$. Since $K \subset X_{a}$ and $X_{a}$ is a closed subspace
of $X$ (cf. Proposition~\ref{pro:closed}), the function $f$ also has absolutely
continuous quasi-norm and therefore $\|f\chi_{E_{k}}\|_{X}\underset{k}{\longrightarrow}0$.
But then \[
\|f_{\sigma(k)}\chi_{E_{\sigma(k)}}\|_{X}\leq C\,(\|f_{\sigma(k)}\chi_{E_{\sigma(k)}}-f\chi_{E_{\sigma(k)}}\|_{X}+\|f\chi_{E_{\sigma(k)}}\|_{X})\underset{k}{\longrightarrow}0,\]
which contradicts \eqref{eq:cont}.\end{proof}
\begin{rem}
\label{rem:remark} Taking into account Remark \ref{rem:not nec},
for future reference we would like to note here that the hypothesis
$K\subset X_{a}$ was not needed in the proof of the assertion that precompactness
in quasi-norm implies local precompactness in measure. 

\medskip{}

The following is an immediate consequence of Theorem \ref{thm:theorem}.
Note that the terminology {}``operator locally compact in measure''
means an operator for which the image of the closed unit ball is precompact in the topology of local convergence in measure.\end{rem}
\begin{cor}
\label{cor:corollary}Let $X$ be a quasi-normed space, $Y$ be q-BFS
and $T$ be a linear operator acting from $X$ into $Y_{a}$. The
operator $T$ is compact if and only if it is locally compact in measure
and $\{Tf:\,\|f\|_{X}\leq1\}\subset\mbox{UAC}(Y)$.
\end{cor}

\section{A compactness criterion in quasi-Banach function spaces over $\Rn$}

\label{sec:4}
\begin{defn}
A function quasi-norm $\rho$ over a totally $\sigma$-finite measure
space $(R,\mu)$ is said to be rearrangement-invariant if $\rho(f)=\rho(g)$
for every pair of equimeasurable functions $f$ and $g$ in $\mathcal{M}_{0}^{+}(R,\mu)$.
The corresponding q-BFS space $X=X(\rho)$ is then said to be \emph{rearrangement-invariant
}(\emph{r.i.}, for short).\end{defn}

In what follows we shall often consider a q-BFS over $\Rn$ with the Lebesgue measure and we shall denote it simply by $L(\Rn)$.

\begin{lem}
\label{lem:minkowski} {\rm (Minkowski-type inequality)} Let $L=L(\Rn)$ be a BFS. If $f\in\mathcal{M}^{+}(\Rn\times\Rn)$ is such
that $\int_{\Rn}\|f(\cdot,y)\|_{L}\, dy<\infty$, then
\[
\left\Vert \int_{\Rn}f(\cdot,y)\, dy\right\Vert _{L}\leq\int_{\Rn}\|f(\cdot,y)\|_{L}\, dy.
\]
\end{lem}
\begin{proof}
Using the properties of the norm in $L$
and Fubini Theorem, we obtain\begin{eqnarray*}
\left\Vert \int_{\Rn}f(\cdot,y)\, dy\right\Vert _{L} & = & \sup_{\|g\|_{L'}\leq1}\Big|\int_{\Rn}\Big(\int_{\Rn}f(x,y)\, dy\Big)g(x)\, dx\Big|\\
 & \leq & \sup_{\|g\|_{L'}\leq1}\int_{\Rn}\Big(\int_{\Rn}|f(x,y)g(x)|\, dx\Big)\, dy\\
 & \leq & \sup_{\|g\|_{L'}\leq1}\int_{\Rn}\|f(\cdot,y)\|_{L}\|g\|_{L'}\, dy\\
 & \leq & \int_{\Rn}\|f(\cdot,y)\|_{L}\, dy\end{eqnarray*}
 \noindent (where $L'$ stands for the associate space of $L$).
\end{proof}
\begin{thm}
\label{thm:compactBFS}Let $L=L(\Rn)$ be a BFS.

\noindent \emph{(a)} If $L$ is r.i. and $K\subset L_{a}$ is precompact in $L$, then:

\emph{(i)} $K$ is bounded in $L$;

\emph{(ii)} $\forall\varepsilon>0,\exists\mbox{compact }G\subset\Rn:\forall u\in K,\,\|u\chi_{\Rn\setminus G}\|_{L}<\varepsilon$;

\emph{(iii)} $\forall\varepsilon>0,\exists\delta>0:\forall u\in K,\,|h|<\delta\Rightarrow\|\Delta_{h}u\|_{L}<\varepsilon$. \footnote{Recall
that $\Delta_{h}u(x):= u(x+h)-u(x), \ x \in \Rn$.}

\noindent \emph{(b)} Conversely, if $K\subset L$ satisfies conditions \emph{(i)-(iii)}, then $K$ is precompact in~$L$.\end{thm}
\begin{proof}
With appropriate modifications, we follow the arguments which prove the well-known
characterization of the compactness in Lebesgue spaces (cf., e.g., \cite[Thm. 2.32]{AdaFou05}).

Suppose first that $K\subset L_{a}$ is precompact. Then, given $\varepsilon>0$, there exists a~finite set $N_{\varepsilon}\subset K$
such that \[
K\subset\bigcup_{f\in N_{\varepsilon}}B_{\varepsilon/6}(f),\]
where $B_{r}(f)$ stands for the open ball in $L$ of radius $r>0$
and centred at $f$. Since $\overline{C_{0}^{\infty}(\Rn)}\supset L_{a}$
(cf. \cite[Rem. 3.13]{EGO97}) and, by hypothesis, $L_{a}\supset K$,
there exists a finite set $S$ of continuous functions with compact
support in $\Rn$ such that for each $u\in K$ there exists $\phi\in S$
satisfying $\|u-\phi\|_{L}<\varepsilon/3$. Let $G$ be the union
of the supports of the finitely many functions from $S$. Then $G\subset\subset\Rn$
and\[
\|u\chi_{\Rn\setminus G}\|_{L}=\|u\chi_{\Rn\setminus G}-\phi\chi_{\Rn\setminus G}\|_{L}\leq\|u-\phi\|_{L}<\varepsilon/3\quad\mbox{for all }\, u\in K.\]

Since continuous functions with compact supports in $\Rn$ are $L$-mean continuous \footnote{Note that a function $g \in L(\Rn)$ is said to be {\it L-mean continuous} 
if for every $\varepsilon >0$ there is a $\delta >0$ such that $\|g(\cdot+h)-g(\cdot)\|_L <\varepsilon$ provided that $h\in \Rn, \ |h|<\delta$.}
 and the set $S$ is finite, there exists $\delta>0$ such that
\[
\|\phi(\cdot+h)-\phi(\cdot)\|_{L}<\varepsilon/3 \quad \mbox{when }\, |h|<\delta \, \mbox{ and }\,\phi\in S.
\]
Then, for any $u\in K$ and $\phi\in S$ such that $\|u-\phi\|_{L}<\varepsilon/3$,\begin{eqnarray*}
\lefteqn{\|u(\cdot+h)-u(\cdot)\|_{L}}\\
 & \leq & \|u(\cdot+h)-\phi(\cdot+h)\|_{L}+\|\phi(\cdot+h)-\phi(\cdot)\|_{L}+\|\phi(\cdot)-u(\cdot)\|_{L}\\
 & < & \varepsilon/3+\varepsilon/3+\varepsilon/3\ =\ \varepsilon\quad\mbox{if }\,|h|<\delta\end{eqnarray*}
(we have also used the fact that $L$ is r.i.).

We have shown that conditions (ii) and (iii) hold. Of course, being precompact, $K$ is also bounded
(in $L$). Therefore (i) is also verified.

We now prove the converse result, i.e., that conditions
(i)-(iii) are sufficient for the precompactness of $K$ in $L$.

Let $\varepsilon>0$ be given. By condition (ii), there is a compact set $G$ in $\Rn$ such that\begin{equation}
\|u\chi_{\Rn\setminus G}\|_{L}<\varepsilon/3\quad\mbox{for all }\, u\in K.\label{eq:tail}\end{equation}

Let $J$ be a non-negative function in $C_{0}^{\infty}(\Rn)$
satisfying $J(x)=0$ if $|x|\geq1$ and $\int_{\Rn}J(x)\, dx=1$. Put
$J_{\eta}(x)=\eta^{-n}J(\eta^{-1}x)$, $x\in\Rn$, $\eta>0$. Then,
for any $u\in L$  (recall that, by \textbf{(P5)}, $u$ is locally integrable), the function $J_{\eta}\ast u$ defined by\[
(J_{\eta}\ast u)(x)=\int_{\Rn}J_{\eta}(x-y)u(y)\, dy=\int_{\Rn}J(y)u(x-\eta y)\, dy\]
belongs
to $C^{\infty}(\Rn)$. In particular, its restriction to $G$ belongs
to $C(G)$.

By Lemma \ref{lem:minkowski} and properties of $J$,
\begin{eqnarray*}
\|J_{\eta}\ast u-u\|_{L} & = & \left\Vert \int_{\Rn}J(y)u(\cdot-\eta y)\, dy-\int_{\Rn}J(y)u(\cdot)\, dy\right\Vert _{L}\\
 & = & \left\Vert \int_{\Rn}\chi_{B(0,1)}(y)J(y)(u(\cdot-\eta y)-u(\cdot)) \, dy\right\Vert _{L}\\
 & \leq & \int_{\Rn}\chi_{B(0,1)}(y)J(y)\|u(\cdot-\eta y)-u(\cdot)\|_{L}\, dy\\
 & \leq & \sup_{y\in B(0,1)}\|u(\cdot-\eta y)-u(\cdot)\|_{L}\\
 & = & \sup_{|h|<\eta}\|u(\cdot-h)-u(\cdot)\|_{L}.\end{eqnarray*}
Thus, according to condition (iii), there is $\delta>0$ such that
\begin{equation}
\|J_{\delta}\ast u-u\|_{L}<\varepsilon/3.\label{eq:parallel}\end{equation}

Now we are going to show that the set $\{(J_{\delta}\ast u)|_{G}:\, u\in K\}$
satisfies the conditions of the Arzelà-Ascoli Theorem in $C(G)$.

First, denoting by $G_{\delta}$ the neighbourhood of radius $\delta$
of $G$ and using conditions \textbf{(P5)} and (i), we arrive at
\begin{eqnarray*}
|(J_{\delta}\ast u)|_{G}(x)| & \leq & \int_{\Rn}J(y)|u(x-\delta y)|\, dy\\
 & \leq & \Big(\sup_{y\in\Rn}J(y)\Big)\,\delta^{-n}\int_{B(x,\delta)}|u(z)|\, dz\\
 & \lesssim & \delta^{-n}\int_{G_{\delta}}|u(y)|\, dy\\
 & \lesssim & 1 \quad\mbox{for all }\, x\in G\,\mbox{ and }\, u\in K.\end{eqnarray*}
Second, 
\begin{eqnarray*}
\lefteqn{|(J_{\delta}\ast u)|_{G}(x+h)-(J_{\delta}\ast u)|_{G}(x)|}\\
 & \leq & \int_{\Rn}J(y)|u(x+h-\delta y)-u(x-\delta y)|\, dy\\
 & \leq & \Big(\sup_{y\in\Rn}J(y)\Big)\,\delta^{-n}\int_{B(x,\delta)}|u(z+h)-u(z)|\, dz\\
 & \lesssim & \delta^{-n}\int_{G_{\delta}}|u(z+h)-u(z)|\, dz\\
 & \lesssim & \|u(\cdot+h)-u(\cdot)\|_{L}\end{eqnarray*}
for all $x,x+h\in G$ and all $u\in L$. Thus, by condition (iii),
given any $\varepsilon_{1}>0$ there exists $\delta_{1}>0$ such that
%\[
%\|\Delta_{h}u\|_{L}<\frac{\varepsilon_{1}}{c(J)\delta^{-n}c_{G_{\delta}}}\quad\mbox{for all }\, u\in K\,\mbox{ and }\,|h|<\delta_{1}.\]
%Consequently, 
\[
|(J_{\delta}\ast u)|_{G}(x+h)-(J_{\delta}\ast u)|_{G}(x)|<\varepsilon_{1}
\]
 for all $u\in K$ and $x,x+h\in G$ with $|h|<\delta_{1}$.

By the Arzelà-Ascoli Theorem, the set $\{(J_{\delta}\ast u)|_{G}:\, u\in K\}$
is precompact in $C(G)$. Therefore, there exists a finite set $\{\psi_{1},\ldots,\psi_{m}\}$
of functions in $C(G)$ with the following property: given any $u\in K$,
there exists $j\in\{1,\ldots,m\}$ such that\begin{equation}
|(J_{\delta}\ast u)(x)-\psi_{j}|<\frac{\varepsilon}{3\|\chi_{G}\|_{L}}\quad\mbox{for all }\, x\in G.\label{eq:arzela}\end{equation}
Thus, denoting by $\tilde{\psi}_{j}$ the extension of $\psi_{j}$
by zero outside $G$, we obtain from \eqref{eq:tail}, \eqref{eq:parallel} and \eqref{eq:arzela} that\begin{eqnarray*}
\|u-\tilde{\psi}_{j}\|_{L} & \leq & \|(u-\tilde{\psi}_{j})\chi_{\Rn\setminus G}\|_{L}+\|(u-\tilde{\psi}_{j})\chi_{G}\|_{L}\\
 & < & \frac{\varepsilon}{3}+\|(u-J_{\delta}\ast u)\chi_{G}\|_{L}+\|(J_{\delta}\ast u-\tilde{\psi}_{j})\chi_{G}\|_{L}\\
 & < & \frac{\varepsilon}{3}+\frac{\varepsilon}{3}+\frac{\varepsilon}{3}\ =\ \varepsilon,\end{eqnarray*}
that is,\[
K\subset\bigcup_{j=1}^{m}B_{\varepsilon}(\tilde{\psi}_{j}),\]
which finishes the proof of the precompactness of $K$ in $L$.\end{proof}
\begin{rem}
\label{rem:coincide} Part (a) of Theorem \ref{thm:compactBFS} remains true if the space $L$ is replaced by any q-BFS $L(\Rn)$ coinciding with an r.i. BFS. Similarly, Part (b) of Theorem~\ref{thm:compactBFS} remains true if the space $L$ is replaced by any q-BFS $L(\Rn)$ coinciding with a~BFS.
\end{rem}
Now we would like to extend Theorem \ref{thm:compactBFS} to a q-BFS more general
than the one considered in Remark \ref{rem:coincide}. To this end, start by
observing that, given a totally $\sigma$-finite measure space $(R,\mu)$,
a function quasi-norm $\rho$ and a positive number $b$, \emph{the
function} $\sigma$ \emph{defined by} \begin{equation}
\sigma(f):=\big(\rho(f^{1/b})\big)^{b},\quad f\in\mathcal{M}^{+}(R,\mu),\label{eq:powerformula}\end{equation}
\emph{is also a function quasi-norm} and the space 
$$X(\sigma) := \{f\in\mathcal{M}(R,\mu):\,\sigma(|f|)<\infty\}$$
is, according to Definition \ref{def:wq-BFS}, the corresponding q-BFS.
Since
\begin{eqnarray*}
\{f\in\mathcal{M}(R,\mu):\,\sigma(|f|)<\infty\} & \!\!\!\! = \!\!\!\! & \{f\in\mathcal{M}(R,\mu):\,\rho(|f|^{1/b})<\infty\}\\
& \!\!\!\! = \!\!\!\! & \{f\in\mathcal{M}(R,\mu):\,|f|=g^{b}\,\mbox{ for some }\, g\in X(\rho)\},
\end{eqnarray*}
\emph{we shall denote} $X(\sigma)$ \emph{also by} $X(\rho)^{b}$, \emph{or simply by}
$X^{b}$ if it is clear that $X$ refers to $X(\rho)$.
\begin{defn}
\label{def:powerqBFS}Given $(R,\mu)$ and a function quasi-norm $\rho$, the space $X(\rho)$
is called a \emph{power q-BFS} if there exists $b\in(0,1]$  such that the space $X(\sigma)$, with $\sigma$ from \eqref{eq:powerformula}, coincides with
a BFS.
\end{defn}
Of course, any BFS is a power q-BFS (take $b=1$ in \eqref{eq:powerformula}).
However, more interesting is to note that a Lebesgue space $L_{p}(\Rn)$
with $0<p<1$ is also a~power q-BFS. Indeed, choosing $b=p$, we obtain
$\|f\|_{X(\sigma)}=\int_{\Rn}|f|$, and thus the space $X(\sigma)$
is the BFS $L_{1}(\Rn)$ (which means that $X(\rho):=L_p(\Rn)$ is a power q-BFS). Another, less trivial, example is given by a~Lorentz space $L_{p,q}(\Rn)$ with $0<p\leq1$
or $0<q\leq1$. In this case we have, for each fixed positive $b$,
\begin{eqnarray}
\|f\|_{X(\sigma)} & = & \||f|^{1/b}\|_{X(\rho)}^{b}\nonumber \\
 & = & \Big(\int_{0}^{\infty}(t^{1/p}f^{\ast}(t)^{1/b})^{q}\,\frac{dt}{t}\Big)^{b/q}\nonumber \\
 & = & \Big(\int_{0}^{\infty}(t^{b/p}f^{\ast}(t))^{q/b}\,\frac{dt}{t}\Big)^{b/q}.\label{eq:lorentznorm}
 \end{eqnarray}
Therefore, $X(\sigma)$ is $L_{p/b,q/b}(\Rn)$,
depending on the choice of $b$. Since this coincides with a
BFS if $p/b,q/b>1$ (cf. \cite[Chap. 4, Thm. 4.6]{BS88}),
the given Lorentz space $L_{p,q}(\Rn)$ is a power q-BFS
(one takes $b \in (0,\min\{p,q\})$).

The notion of a power q-BFS is closely related to the so-called lattice convexity. We recall the latter in the following definition.
\begin{defn}
(see, e.g., \cite[Def. 1.d.3]{LT78} or \cite{Kal84}) If $X=X(\rho)$ is a q-BFL and $b\in(0,\infty)$, then $X$ is said to 
be \emph{b-convex} if for some $C\geq0$ and any $g_{1},\ldots,g_{m}\in X$,\[
\left\Vert \left(\sum_{i=1}^{m}|g_{i}|^{b}\right)^{1/b}\right\Vert _{X}\leq C\left(\sum_{i=1}^{m}\|g_{i}\|_{X}^{b}\right)^{1/b}.\]
 We have the following result.\end{defn}
\begin{prop}
\label{equivbconvex}
Let $X=X(\rho)$ be a q-BFL and $b\in(0,\infty)$. Then the functional $f\mapsto\||f|^{1/b}\|_{X}^{b}$
is equivalent to a norm in $X^{b}$ if and only if $X$ is $b$-convex.
In particular, given a q-BFS $X$, then $X$ is a power q-BFS if and only if there is $b\in(0,1]$ such that $X$ is a b-convex q-BFL.
\end{prop}
\begin{proof}
Assuming that the functional $f\mapsto\||f|^{1/b}\|_{X}^{b}$ is equivalent to a norm in $X^{b}$,
the $b$-convexity of $X$ follows immediately. 

On the contrary, assume now that $X$ is a $b$-convex q-BFL and
define\[
|||f|||:=\inf\left(\sum_{i=1}^{m}\||f_{i}|^{1/b}\|_{X}^{b}\right)\quad\mbox{ for any }\, f\in X^{b},\]
where the infimum is taken over all possible decompositions $f=\sum_{i=1}^{m}f_{i}$
with $f_{i}\in X^{b}$ and $m\in\N$. Clearly, $|||f|||\leq\||f|^{1/b}\|_{X}^{b}$
for all $f\in X^{b}$. As for the reverse inequality, given any $f\in X^{b}$
and any $\varepsilon>0$, there exists a decomposition $f=\sum_{i=1}^{m}f_{i}$
such that \[
\sum_{i=1}^{m}\||f_{i}|^{1/b}\|_{X}^{b}\leq|||f|||+\varepsilon.\]
Therefore, using the $b$-convexity of $X$,\[
\||f|^{1/b}\|_{X}^{b}\leq\left\Vert \left(\sum_{i=1}^{m}|f_{i}|\right)^{1/b}\right\Vert _{X}^{b}\leq C^{b}\left(\sum_{i=1}^{m}\||f_{i}|^{1/b}\|_{X}^{b}\right)\leq C^{b}|||f|||+C^{b}\varepsilon.\]
Since $\varepsilon>0$ is arbitrary, we conclude that $\||f|^{1/b}\|_{X}^{b}\leq C^{b}|||f|||.$

So we have proved that the two expressions $\||f|^{1/b}\|_{X}^{b}$
and $|||f|||$ are equivalent. We leave to the reader to verify that the functional $|||\cdot|||$ is a norm in $X^{b}$.
%Consequently, $f\in X^{b}$ is the zero function if and only if $|||f|||=0$.
%Hence, it only remains to show the homogeneity and the triangle inequality
%of $|||\cdot|||.$
%
%Homogeneity: Given any $f\in X^{b}$ and any $\alpha\in\C$,\begin{eqnarray*}
%|||\alpha f||| & = & \inf_{\alpha f=\sum_{i=1}^{m}h_{i}}\left(\sum_{i=1}^{m}\||h_{i}|^{1/b}\|_{X}^{b}\right)\\
 %& = & \inf_{f=\sum_{i=1}^{m}f_{i}}\left(\sum_{i=1}^{m}\||\alpha f_{i}|^{1/b}\|_{X}^{b}\right)\\
 %& = & \alpha\,|||f|||.\end{eqnarray*}
%
%
%Triangle inequality: Given any $f,g\in X^{b}$ and any $\varepsilon>0$,
%there are decompositions $f=\sum_{i=1}^{m}f_{i}$ and $g=\sum_{j=1}^{\ell}g_{j}$
%such that \[
%\sum_{i=1}^{m}\||f_{i}|^{1/b}\|_{X}^{b}\leq|||f|||+\varepsilon,\quad\sum_{j=1}^{\ell}\||g_{j}|^{1/b}\|_{X}^{b}\leq|||g|||+\varepsilon.\]
%Therefore, \begin{eqnarray*}
%|||f+g||| & = & \inf_{f+g=\sum_{k=1}^{n}h_{k}}\left(\sum_{k=1}^{n}\||h_{k}|^{1/b}\|_{X}^{b}\right)\\
 %& \leq & \sum_{i=1}^{m}\||f_{i}|^{1/b}\|_{X}^{b}+\sum_{j=1}^{\ell}\||g_{j}|^{1/b}\|_{X}^{b}\\
 %& \leq & |||f|||+|||g|||+2\varepsilon.\end{eqnarray*}
%Since $\varepsilon>0$ is arbitrary, we conclude that $|||f+g|||\leq|||f|||+|||g|||$.
\end{proof}
\begin{rem}
Note that criteria for $b$-convexity of some function lattices $X$
have been studied -- see, e.g., \cite{KamMal04}, \cite{KM04}
for the case when $X$ is one of the Lorentz spaces $\Lambda^{q}(\omega)$ or $\Gamma^{q}(\omega)$, $0<q<\infty$.
\end{rem}
\smallskip{}

\begin{thm}
\label{thm:compactqBFS}Let $L=L(\Rn)$ be a power q-BFS.

\noindent \emph{(a)} If $L$ coincides with a rearrangement invariant q-BFS and $K\subset\overline{C_{0}(\Rn)}$ is precompact in $L$, then $K$
satisfies conditions \emph{(i)-(iii)} of \emph{Theorem \ref{thm:compactBFS}}.

\noindent \emph{(b)} Conversely, if $K\subset L_{a}$ satisfies conditions \emph{(i)-(iii)} of \emph{Theorem \ref{thm:compactBFS}},
then $K$ is precompact in $L$.
\end{thm}
\begin{proof}
With the hypotheses assumed, the proof of the necessity of conditions
(i)-(iii) of Theorem \ref{thm:compactBFS} essentially follows the corresponding part of the proof of that theorem, with
slight modifications. Therefore, we prove only that these conditions
are sufficient for the precompactness of $K\subset L_{a}$ in $L$ when $L$ is a~power q-BFS. 
Since the result follows immediately from Remark \ref{rem:coincide}
if $b=1$ in Definition~\ref{def:powerqBFS}, suppose that $b\in(0,1)$. 

We shall assume first that $K$ contains only real
functions. Denote 
\begin{equation}
K_+^{b} := \{f_{+}^{b}:\, f\in K\},
\label{eq:4sets}
\end{equation}
where the symbol $+$ in subscript indicates positive parts of functions and sets. It is easy
to see that $K_+^{b}$ is a subset of $L^{b}$, the latter coinciding,
by Definition~\ref{def:powerqBFS}, with a BFS. We shall show
that conditions (i)-(iii) of Theorem \ref{thm:compactBFS} hold for
$K_+^{b}$ and $L^{b}$ instead of $K$ and $L$, respectively. Therefore, by Remark \ref{rem:coincide}, we conclude 
that $K_+^{b}$ is precompact in $L^{b}$. Indeed, condition (i) for $K_+^{b}$ and $L^{b}$ follows from condition
(i) for $K$ and $L$, since
\[
\|f_{+}^{b}\|_{L^{b}} := \| (f_{+}^{b})^{1/b} \|_L^b = \| f_{+} \|_L^b \leq\|f\|_{L}^{b}.
\]
Similarly, condition (ii) for $K_+^{b}$ and $L^{b}$ follows from condition (ii)
for $K$ and $L$, due to the inequality
\begin{equation}
\|\left(f_{+}^{b}\right)\chi_{\Rn\setminus G}\|_{L^{b}} \leq \|f\chi_{\Rn\setminus G}\|_{L}^{b}.\label{eq:inequalities}
\end{equation}
Finally, condition (iii) for $K_+^{b}$ and $L^{b}$ follows from the
corresponding condition for $K$ and $L$ and the estimate
\begin{eqnarray*}
\|f_+^{b}(\cdot+h)-f_+^{b}(\cdot)\|_{L^{b}} & \leq & \||f_+(\cdot+h)-f_+(\cdot)|^{b}\|_{L^{b}}\\
 & = & \|f_+(\cdot+h)-f_+(\cdot)\|_{L}^{b}\\
 & \leq & \|f(\cdot+h)-f(\cdot)\|_{L}^{b};
\end{eqnarray*}
the first inequality is a consequence of the fact that $b\in(0,1)$.
Hence, as mentioned above, $K_+^{b}$ is precompact in $L^{b}$.

Similarly, we can conclude that $K_-^b$ is precompact in $L^b$, where
\begin{equation}
K_-^{b} := \{f_{-}^{b}:\, f\in K\},
\label{eq:Kminusb}
\end{equation}
the symbol $-$ in subscript indicating negative parts of functions and sets.

Now take any sequence $\{f_{k}\}_{k}\subset K$ and consider the sequence
\begin{equation}
\{(f_{k})_{+}^{b}\}_{k} \subset K_+^b.
\label{eq:seqK-b}
\end{equation}

Due to the precompactness of $K_+^{b}$ in $L^{b}$, there is a subsequence
$\{(f_{\tau_1(k)})_{+}^{b}\}_{k}$ converging in $L^{b}$, say to $g$.
Consider an increasing sequence $\{E_{m}\}_{m\in\N}$ of compact subsets
of $\Rn$ such that $\cup_{m\in\N}E_{m}=\Rn$. Since $L^{b}$ coincides
with a BFS, denoting by $\lambda$ the Lebesgue measure in $\Rn$
and using \textbf{(P5)}, we obtain that  
\begin{equation}
\int_{E_{m}}|(f_{\tau_1(k)})_{+}^{b}-g|\, d\lambda\underset{k}{\longrightarrow}0 \quad \mbox{for any } m\in\N.
\label{eq:L1conv}
\end{equation}
%
\begin{comment}
Clearly, $g=g_{+}-g_{-}$ and\begin{eqnarray*}
0 & \leq & \int_{E_{m}\cap\{g_{-}>0\}}g_{-}\, d\lambda\\
 & \leq & \int_{E_{m}\cap\{g_{-}>0\}}|g_{\tau_{1}(k)}^{b}+g_{-}|\, d\lambda+\int_{E_{m}\cap\{g_{-}=0\}}|g_{\tau_{1}(k)}^{b}-g_{+}|\, d\lambda\\
 & = & \int_{E_{m}}|g_{\tau_{1}(k)}^{b}-g|\, d\lambda\;\underset{k}{\longrightarrow}0,\end{eqnarray*}
therefore $g_{-}=0$ $\,\lambda$-a.e. and we can, instead of \eqref{eq:L1conv},
write, more precisely,\begin{equation}
\int_{E_{m}}|g_{\tau_{1}(k)}^{b}-g_{+}|\, d\lambda\underset{k}{\longrightarrow}0.\label{eq:L1conv+}\end{equation}
\end{comment}
%
\noindent Consequently, there are subsequences $\{(f_{(\tau_2 \circ \tau_1)(k)})_{+}^{b}\}_{k}$,
$\{(f_{(\tau_3 \circ \tau_2 \circ \tau_1)(k)})_{+}^{b}\}_{k}$, $\ldots$
converging $\lambda$-a.e. to $g$, respectively in $E_{2}$, $E_{3}$,
$\ldots$. Then the diagonal sequence 
\[
(f_{(\tau_2 \circ \tau_1)(2)})_{+}^{b},\, (f_{(\tau_3 \circ \tau_2 \circ \tau_1)(3)})_{+}^{b},\, (f_{(\tau_4 \circ \tau_3 \circ \tau_2 \circ \tau_1)(4)})_{+}^{b},\,\ldots
\]
(which is a subsequence of $\{(f_{k})_{+}^{b}\}_{k}$) converges $\lambda$-a.e.
to $g$. Denote it simply by $\{(f_{\tau(k)})_{+}^{b}\}_{k}$. Consequently,
\begin{equation}
(f_{\tau(k)})_{+}\underset{k}{\longrightarrow}|g|^{1/b}\quad\lambda\mbox{-a.e.}.\label{eq:limitreal}
\end{equation}
We recall that $g\in L^{b}$ therefore $|g|^{1/b}\in L$. 
Now, repeating the procedure above, but starting with $\{(f_{\tau(k)})_{-}^{b}\}_{k} \subset K_-^b$ instead of (\ref{eq:seqK-b}), we arrive at
\begin{equation}
(f_{\sigma(k)})_{-}\underset{k}{\longrightarrow}|h|^{1/b}\quad\lambda\mbox{-a.e.}, \quad \mbox{where }\, h \in L^b. \label{eq:limitreal-}
\end{equation}
Conclusions \eqref{eq:limitreal} and \eqref{eq:limitreal-} imply
\[
f_{\sigma(k)}\underset{k}{\longrightarrow}|g|^{1/b}-|h|^{1/b}\quad\lambda\mbox{-a.e.}.
\]
According to what was mentioned above, $|g|^{1/b}-|h|^{1/b}\in L$.
Since convergence $\lambda\mbox{-a.e.}$ yields local convergence
in measure, we have just proved that 
\begin{equation}
K \;\; \mbox{ is locally precompact in measure.}
\label{eq:precompact}
\end{equation}
The hypothesis $K\subset L_{a}$ implies
that both $K_+^{b}$ and $K_-^b$ (given, respectively, by \eqref{eq:4sets} and \eqref{eq:Kminusb}) are subsets of
$(L^{b})_{a}$. Together with the conclusion
that $K_+^{b}$ and $K_-^b$ are precompact in $L^{b}$ and Theorem \ref{thm:theorem}, this yields $K_+^{b}, K_-^b\subset\mbox{UAC}(L^{b})$. Since also
\begin{eqnarray*}
\|f\chi_{E_{m}}\|_{L} & = & \|\big(f_{+}\big)\chi_{E_{m}}-\big(f_{-}\big)\chi_{E_{m}}\|_{L}\\
 & \lesssim & \|\big(f_{+}\big)\chi_{E_{m}}\|_{L}+\|\big(f_{-}\big)\chi_{E_{m}}\|_{L}\\
 & = & \|\big(f_{+}^{b}\big)\chi_{E_{m}}\|_{L^{b}}^{1/b}+\|\big(f_{-}^{b}\big)\chi_{E_{m}}\|_{L^{b}}^{1/b}\end{eqnarray*}
for any $f\in K$ and any sequence $\{E_{m}\}_{m}$ of $\lambda$-measurable subsets of $\Rn$ with $E_{m}\underset{m}{\longrightarrow}\emptyset\;\lambda\mbox{-a.e.}$, we see that $K\subset\mbox{UAC}(L)$. 
This fact, \eqref{eq:precompact} and Theorem \ref{thm:theorem} imply that $K$ is precompact in $L$.

In the case in which $K$ contains complex functions, we get the result applying the above method successively to the real parts and then to the imaginary parts of functions.
\end{proof}
\begin{rem}
\label{rem:compactqBFS} In part (a) of Theorem \ref{thm:compactqBFS} one can use the assumptions
that $L^b$ coincides with an r.i. BFS, $L=L_{a}$ and $K\subset L$ instead of the hypotheses that $L$ coincides
with an r.i. q-BFS and $K\subset\overline{C_{0}(\Rn)}$. Indeed, from the assumptions made now one can prove that $\overline{C_{0}(\Rn)}=L \supset K$, with arguments similar to the ones used above (replacing Theorem \ref{thm:theorem} by Lemma \ref{lem:lemma}), and that $L$ coincides with an r.i. q-BFS.
\end{rem}

\section{Abstract Besov spaces and compact embeddings}

\label{sec:5}

In what follows, $\mathbb{R}^{n}$ and $(0,1)$ are endowed with the
corresponding Lebesgue measures. The \emph{modulus of continuity}
of a function $f$ in a q-BFS $L=L(\Rn)$ over $\mathbb{R}^{n}$
is given by 

  \[     \omega_L(f,t) := \sup_{{h \in \Rn} \atop {| h | \leq t}} \| \Delta_h f \|_L, \quad t>0. \]

The following is an extension of the definition of generalized Besov spaces  considered in \cite{GPS12} to the setting of quasi-Banach function spaces. 

\begin{defn}
\label{def:Besov}Let $L=L(\Rn)$ be a q-BFS and let $Y$ be a q-BFL over $(0,1)$ satisfying
\begin{equation}
\|1\|_{Y}=\|\chi_{(0,1)}\|_{Y}=\infty.\label{eq:assumpinfty}\end{equation}
The abstract Besov space $B(L,Y)$ is the set of all $f\in L$
such that the quasi-norm\[
\|f\|_{B(L,Y)}:=\|f\|_{L}+\|\omega_{L}(f,\cdot)\|_{Y}\]
is finite.\end{defn}
\begin{rem}
\label{rem:varphiY-1} The assumption $\|1\|_{Y}=\infty$
in the Definition \ref{def:Besov} is natural (otherwise $B(L,Y)=L$, which is not of interest).
Also notice that assumption \eqref{eq:assumpinfty} violates axiom
\textbf{(P4)}. Consequently, $Y$ is not a q-BFS.

\medskip{}

In what follows we consider the space 
$$L(\Omega):=\{f|_{\Omega}:\, f\in L\},$$
for a measurable set $\Omega\subset\Rn$ with $|\Omega|_{n}>0$, 
quasi-normed by
\[
\|f|_{\Omega}\|_{L(\Omega)}:=\|f\chi_{\Omega}\|_{L}.\]
\end{rem}
\begin{thm}
\label{thm:compactinL}Let $L=L(\Rn)$ be a power q-BFS.
In the case that $L$ does not coincide with a BFS, we assume
that $L=L_{a}$. Let $B(L,Y)$ be an abstract Besov space and let $\Omega$ be a bounded measurable
subset of $\Rn$ with $|\Omega|_{n}>0$. Then 
\[
B(L,Y)\hookrightarrow\hookrightarrow L(\Omega)\]
{\rm (}this means that the restriction operator $f\mapsto f|_{\Omega}$
is compact from $B(L,Y)$ into $L(\Omega)${\rm )}.\end{thm}
\begin{proof}
We are going to prove that
\[
\{f|_{\Omega}:\, f\in B(L,Y), \|f\|_{B(L,Y)}\leq1\}\]
is precompact in $L(\Omega)$.

Since $\Omega$ is bounded, there is $R_{0}\in(0,\infty)$ such that
$\overline{\Omega}\subset B(0,R_{0})$. Take a~Lipschitz continuous
function $\varphi$ on $\Rn$ satisfying $0\leq\varphi\leq1$, $\varphi=1$
on $\Omega$ and $\varphi=0$ on $\Rn\setminus B(0,R_{0}+1)$. Then, for all
$f\in L$,
\[
\|f|_{\Omega}\|_{L(\Omega)}=\|f\chi_{\Omega}\|_{L}=\|\varphi f\chi_{\Omega}\|_{L}\leq\|\varphi f\|_{L}.\]
Therefore, it is sufficient to prove that the set\[
K:=\{\varphi f:\, f\in B(L,Y),\|f\|_{B(L,Y)}\leq1\}\]
is precompact in $L$.

By Theorems \ref{thm:compactBFS}, \ref{thm:compactqBFS} and Remark
\ref{rem:coincide}, it is enough to verify that conditions (i)-(iii)
of Theorem \ref{thm:compactBFS} hold for the $K$ and $L$ considered
here.

(i) Let $\varphi f\in K$, with $\|f\|_{B(L,Y)}\leq1$. Since $\|\varphi f\|_{L}\leq\|f\|_{L}\leq\|f\|_{B(L,Y)}$,
the set $K$ is bounded in $L$.

(ii) The set $G:=B(0,R_{0}+1)$ is compact in $\Rn$. If
$\varphi f\in K$, then $\|\varphi f\chi_{\Rn\setminus G}\|_{L}=0<\varepsilon$
for all $\varepsilon>0$.

(iii) Given $f\in L$ and $x,h\in\Rn$, \begin{eqnarray}
|\Delta_{h}(\varphi f)(x)| & \leq & \|\varphi\|_{\infty,\Rn}|\Delta_{h}f(x)|+\|\Delta_{h}\varphi\|_{\infty,\Rn}|f(x)|\nonumber \\
 & \lesssim & |\Delta_{h}f(x)|+|h||f(x)|.\label{eq:differences}\end{eqnarray}
Hence, \[
\|\Delta_{h}(\varphi f)\|_{L}\lesssim\|\Delta_{h}f\|_{L}+|h|\|f\|_{L}.\]
If $\varphi f\in K$, with $\|f\|_{B(L,Y)}\leq1$,
then conditions \textbf{(P2)}, \textbf{(P1)}, together with the facts $Y \subset \mathcal{M}_0(0,1)$ (recall Lemma \ref{star}) and that $\omega_{L}(f,\cdot)$ is non-decreasing, imply that 
\begin{equation}
1\geq\|f\|_{B(L,Y)}\geq\|\omega_{L}(f,\cdot)\chi_{(T,1)}\|_{Y}\geq\omega_{L}(f,T)\|\chi_{(T,1)}\|_{Y}\quad \mbox{for } T\in(0,1).
\label{eq:less1}
\end{equation}
Moreover, \textbf{(P3)} and \eqref{eq:assumpinfty} yield that $\lim_{T\to0+}\|\chi_{(T,1)}\|_{Y}$ $=$ $\infty$. Consequently, we obtain from \eqref{eq:less1} that
\[
\lim_{T\to0+}\omega_{L}(f,T)=0\quad\mbox{uniformly with respect to }f\,\mbox{ such that }\,\varphi f\in K.
\]
Combining with \eqref{eq:differences} and recalling that $\|f\|_{L}\leq\|f\|_{B(L,Y)}$,
we get that\[
\forall\varepsilon>0,\exists\delta>0:\forall(\varphi f)\in K,\,|h|<\delta\,\Rightarrow\|\Delta_{h}(\varphi f)\|_{L}<\varepsilon.\]

Applying Theorem \ref{thm:compactBFS} and Remark \ref{rem:coincide},
or Theorem \ref{thm:compactqBFS}, we obtain the precompactness of
$K$ in $L$ and the proof is complete.
\end{proof}
\begin{cor}
\label{cor:KsubUAC(Z)}Let $L=L(\Rn)$ be a power q-BFS.
In the case that $L$ does not coincide with a BFS, assume 
that $L=L_{a}$. Let $B(L,Y)$ be an abstract Besov space. Let $\Omega$ be a bounded measurable
subset of $\Rn$ with $|\Omega|_{n}>0$ and let $Z(\Omega)$ be a q-BFS
over $\Omega$ $($with the restricted Lebesgue measure$)$. Assume that
the restriction operator maps $B(L,Y)$ into $(Z(\Omega))_{a}$.
Then this operator is compact if and only if \begin{equation}
K:=\{f|_{\Omega}:\, f\in B(L,Y), \|f\|_{B(L,Y)}\leq1\}\subset \mbox{UAC}(Z(\Omega)).\label{eq:KsubUAC}\end{equation}
\end{cor}
\begin{proof}
As a consequence of Theorem \ref{thm:compactinL}, the set $K$ in \eqref{eq:KsubUAC}
is precompact in $L(\Omega)$. Therefore, by Theorem \ref{thm:theorem}
and Remark \ref{rem:remark}, such a $K$ is locally precompact in
measure in $\mathcal{M}_{0}(\Omega)$. Hence, by Corollary \ref{cor:corollary}, the set $K$ is precompact
in $Z(\Omega)$ if and only if $K\subset \mbox{UAC}(Z(\Omega))$.\end{proof}
\begin{rem}
\label{rem:KsubUAC(Z)}The assumptions for the {}``if'' part of the Corollary \ref{cor:KsubUAC(Z)} can be relaxed. In fact, from Definition \ref{def:unifabscont} and Theorem
\ref{thm:theorem}, we have that if \eqref{eq:KsubUAC} holds, then
the restriction operator takes $B(L,Y)$ compactly into $Z(\Omega)$.
That is, we don't need to assume \emph{a priori} that the restrictions
of the elements of $B(L,Y)$ are in $(Z(\Omega))_{a}$ (since this
is a consequence of \eqref{eq:KsubUAC}, by the
homogeneity of the quasi-norm). Actually, here we don't even have
to assume \emph{a priori} that those restrictions belong to $Z(\Omega)$
(again, this is a consequence of \eqref{eq:KsubUAC}).
\end{rem}

\section{Applications}

\label{sec:6}

We are going to apply Corollary \ref{cor:KsubUAC(Z)} and Remark
\ref{rem:KsubUAC(Z)} to the case when $Z(\Omega)$ is the
Lorentz-type space $L_{p,q;w}(\Omega)$ introduced in Section \ref{sec:2}.
We start with the following criterion for a subset of measurable
functions to have uniformly absolutely continuous quasi-norm in $L_{p,q;w}(\Omega)$.
\begin{prop}
\label{pro:limsupint}Let $0<p,q\leq\infty$, let $\Omega$
be a measurable subset of $\Rn$ with $0<|\Omega|_{n}<\infty$ and
$w\in\mathcal{W}(0,|\Omega|_{n}).$ Assume that \eqref{eq:EGO(2.1)}
and \eqref{eq:EGO(2.5)} hold. If $K\subset\mathcal{M}(\Omega)$ is
such that
\begin{equation}
\lim_{\delta\to0+}\sup_{u\in K} \| t^{1/p-1/q}w(t)u^{*}(t) \|_{q;(0,\delta)} = 0,\label{eq:limsup}
\end{equation}
then $K\subset \mbox{UAC}(L_{p,q;w}(\Omega))$.\end{prop}
\begin{proof}
Given any $\delta\in(0,|\Omega|_{n})$ and $u \in K$,
\begin{eqnarray*}
\lefteqn{\|t^{1/p-1/q}w(t)u^{*}(t)\|_{q;(0,|\Omega|_{n})}}\\
 & \lesssim & \|t^{1/p-1/q}w(t)u^{*}(t)\|_{q;(0,\delta)}+u^{*}(\delta)\|t^{1/p-1/q}w(t)\|_{q;(\delta,|\Omega|_{n})}\\
 & \leq & \|t^{1/p-1/q}w(t)u^{*}(t)\|_{q;(0,\delta)}+u^{*}(\delta)B_{p,q;w}(|\Omega|_{n}) =: V(\delta,u).
\end{eqnarray*}
Combining with \eqref{eq:limsup} and \eqref{eq:EGO(2.1)}, one can choose $\delta\in(0,|\Omega|_{n})$ such that $V(\delta,u)$ is finite. Therefore, $K\subset L_{p,q;w}(\Omega)$.
Moreover, given any $\varepsilon>0$, property \eqref{eq:limsup} implies
that there is $\delta>0$ such that for all $u\in K$ and all measurable
$E$, with $|E|_{n}<\delta$, \[
\|u\chi_{E}\|_{p,q;w;\Omega}\leq \| t^{1/p-1/q}w(t)u^{*}(t)\chi_{[0,\delta)}(t) \|_{q;(0,|\Omega|_{n})} < \varepsilon,\]
which proves that $K\subset \mbox{UAC}(L_{p,q;w}(\Omega))$, due to Remark
\ref{rem:converse}(i).
\end{proof}

%Now we would like to apply Corollary \ref{cor:KsubUAC(Z)} and Remark
%\ref{rem:KsubUAC(Z)} to the Besov space $B_{p,q}^{s}(\Rn)$, with
%$0<s<1$ and $0<p,q\leq\infty$, defined as the set of all functions
%$f\in L_{p}(\Rn)$ such that the quasi-norm\[
%\|f\|_{B_{p,q}^{s}(\Rn)}:=\|f\|_{p}+\|t^{-s-1/q}\omega_{L_{p}(\Rn)}(f,t)\|_{q;(0,1)}\]
%is finite. It is easy to see that this fits in Definition \ref{def:Besov}
%with the choice $L=L_{p}(\Rn)$ and $Y=\{g\in\mathcal{M}(0,1):\,\|g\|_{Y}=\|t^{-s-1/q}g(t)\|_{q;(0,1)}<\infty\}$.

Let $0<s<1$ and $0<p,q\leq\infty$. Choosing $L=L_{p}(\Rn)$ and $Y=\{g\in\mathcal{M}(0,1):\,\|g\|_{Y}=\|t^{-s-1/q}g(t)\|_{q;(0,1)}<\infty\}$ in Definition \ref{def:Besov}, we obtain the well-known Besov spaces (of small smoothness $s$, defined by differences), equipped with the quasi-norm
\[
\|f\|_{B_{p,q}^{s}(\Rn)}:=\|f\|_{p}+\|t^{-s-1/q}\omega_{L_{p}(\Rn)}(f,t)\|_{q;(0,1)}.
\]

%Compact embeddings of this space into a Lorentz-type space have been
%obtained when $p,q\geq1$: see ... %
%\marginpar{insert references%}. 
Recently the so-called growth envelope functions for such spaces have been obtained in \cite[Prop. 2.5, Thm. 2.7]{HarSch09}. We recall the result.
%which, together with the tools developed in the present paper, allow
%us to deal easily with questions of compactness of such Besov spaces
%into Lorentz-type spaces, including for positive values of $p$ or
%$q$ less than 1. We first recall such growth envelope functions results.
\begin{prop}
\label{pro:growthenv}Let $0<s<1$, $0<p,q\leq\infty$. 

\emph{(i)} If $s<\frac{n}{p}$, then \[
\sup_{\|f\|_{B_{p,q}^{s}(\Rn)}\leq1}f^{*}(t)\approx t^{-1/p+s/n}\quad\mbox{as }t\to0+.\]

\emph{(ii)} If $s=\frac{n}{p}$ and $q>1$, then\[
\sup_{\|f\|_{B_{p,q}^{s}(\Rn)}\leq1}f^{*}(t)\approx|\log t|^{1/q'}\quad\mbox{as }t\to0+.\]

\emph{(iii)} In all the remaining cases,\[
\sup_{\|f\|_{B_{p,q}^{s}(\Rn)}\leq1}f^{*}(t)\approx1\quad\mbox{as }t\to0+.\]

\end{prop}

Using Proposition \ref{pro:growthenv} and our results, we obtain compact embeddings of these Besov spaces.

\begin{thm}
\label{Besovcompact}
Let $0<s<1$ and $0<p,q,r,u\leq\infty$. Let $\Omega$
be a~bounded measurable subset of $\Rn$, $|\Omega|_{n}>0$, and
$w\in\mathcal{W}(0,|\Omega|_{n}).$ Assume that \eqref{eq:EGO(2.1)}, \eqref{eq:EGO(2.5)} and one of the following
conditions are satisfied:

\emph{(i)} $s<\frac{n}{p}\,$ and $\,\lim_{\delta\to0+}\|t^{1/r-1/u-1/p+s/n}w(t)\|_{u;(0,\delta)}=0$;

\emph{(ii)} $s=\frac{n}{p}$, $q>1\,$ and $\,\lim_{\delta\to0+}\|t^{1/r-1/u}|\log t|^{1/q'}w(t)\|_{u;(0,\delta)}=0$;

\emph{(iii)} $s>\frac{n}{p}$, or $s=\frac{n}{p}$ and $0<q\leq1$, and $\,\lim_{\delta\to0+}\|t^{1/r-1/u}w(t)\|_{u;(0,\delta)}=0$. 

Then \[
B_{p,q}^{s}(\Rn)\hookrightarrow\hookrightarrow L_{r,u;w}(\Omega).\]
\end{thm}
\begin{proof}
Observing that\begin{eqnarray*}
\lefteqn{0\leq\sup_{\|f\|_{B_{p,q}^{s}(\Rn)}\leq1} \| t^{1/r-1/u}w(t)(f|_{\Omega})^{*}(t) \|_{u,(0,\delta)}}\\
 & \leq & \| t^{1/r-1/u}w(t)\sup_{\|f\|_{B_{p,q}^{s}(\Rn)}\leq1}f^{*}(t) \|_{u;(0,\delta)},\end{eqnarray*}
the conclusion follows from Propositions \ref{pro:growthenv}, \ref{pro:limsupint}, Corollary
\ref{cor:KsubUAC(Z)} and Remark~\ref{rem:KsubUAC(Z)}.
\end{proof}
\begin{rem}
When more precise estimates than the growth envelope functions are known, our approach may lead to weaker sufficient conditions for compactness. For example, when $s<n/p$ (with $0<s<1$) and
$0<q\leq u\leq \infty$, we have, for small $\delta>0$, from \cite[Cor. 3.3]{HarSch09} that
\[
\| t^{1/p-s/n-1/u}f^{*}(t) \|_{u;(0,\delta)} \lesssim\|f\|_{B_{p,q}^{s}(\Rn)}\quad\mbox{for all }\, f\in B_{p,q}^{s}(\Rn).
\]
Consequently,
\[
\sup_{\|f\|_{B_{p,q}^{s}(\Rn)}\leq1} \| t^{1/p-s/n-1/u}f^{*}(t) \|_{u;(0,\delta)} \lesssim1\quad\mbox{for all small }\,\delta>0.\]
Therefore, with $\frac{1}{r}:=\frac{1}{p}-\frac{s}{n}$, 
\begin{eqnarray*}
0 & \leq & \sup_{\|f\|_{B_{p,q}^{s}(\Rn)}\leq1} \| t^{1/r-1/u}w(t)(f|_{\Omega})^{*}(t) \|_{u;(0,\delta)}\\
 & \lesssim & \mbox{ess sup}_{t\in(0,\delta)}w(t).
 \end{eqnarray*}
Hence, the assumption $\lim_{t\to0+}w(t)=0$ (for the mentioned
conjugation of parameters) guarantees that $B_{p,q}^{s}(\Rn)\hookrightarrow\hookrightarrow L_{r,u;w}(\Omega)$.
\end{rem}

\section*{Acknowledgements:}

We thank the Departments of Mathematics of the University of Coimbra for a~kind hospitality during our stays there.

\noindent
We also thank a referee for valuable comments.

%\bibliographystyle{plain}
%\bibliography{C:/localtexmf/bibtex/bib/base/bibliography}

\end{document}